\documentclass[11pt]{amsart}

\usepackage{hyperref}
\usepackage{amsfonts}
\usepackage{graphicx}
\usepackage{tabularx}
\usepackage{array}
\usepackage[usenames,dvipsnames]{color}
\usepackage{comment}
\usepackage{amsmath}
\usepackage{amsthm}
\usepackage{amssymb}
\usepackage{fullpage}
\usepackage[dvipsnames]{xcolor}
\usepackage{tikz}
\usepackage{listings}
\usepackage{dsfont}
\usepackage{enumerate}
\usepackage{romannum}

\newtheorem{theorem}{Theorem}[section]

\newtheorem{conjecture}[theorem]{Conjecture}
\newtheorem{problem}[theorem]{Problem}

\newtheorem{lemma}[theorem]{Lemma}
\newtheorem{corollary}[theorem]{Corollary}
\theoremstyle{definition}

\def\epsilon{\varepsilon}

\title{
On Gy\'{a}rf\'{a}s' Path-Colour Problem
}
\author{Ben Cameron}
\address{School of Mathematical and Computational Sciences, University of PEI}
\email{brcameron@upei.ca}
\author{Alexander Clow}
\address{Department of Mathematics, Simon Fraser University}
\email{alexander\_clow@sfu.ca}
\date{\today}

\begin{document}
\pagenumbering{arabic}

\begin{abstract}
    In their 1997 paper titled ``Fruit Salad", Gy\'{a}rf\'{a}s posed the following conjecture:
    there exists a constant $k$ such that if each path of a graph spans a $3$-colourable subgraph, then the graph is $k$-colourable.
    It is noted that $k=4$ might suffice.
    Let $r(G)$ be the maximum chromatic number of any subgraph $H$ of $G$ where $H$ is spanned by a path.
    The only progress on this conjecture comes from Randerath and Schiermeyer in 2002, who proved that if $G$ is an $n$ vertex graph, then $\chi(G) \leq r(G)\log_{\frac{8}{7}}(n)$.
    
    Gy\'{a}rf\'{a}s notes this conjecture is a weakened version of the following open problem of Erd\H{o}s and Hajnal, which Erd\H{o}s shared with Gy\'{a}rf\'{a}s in 1995:
    Is there a function $f$ such that if each odd cycle of a graph spans a subgraph with chromatic number at most $r$, then the chromatic number the graph is at most $f(r)$?

    We prove that for all natural numbers $r$, there exists a graph $G$ with $r(G)\leq r$ and $\chi(G)\geq \lfloor\frac{3r}{2}\rfloor -1$.
    Hence, for all constants $k$ there exists a graph with $\chi - r > k$.
    Our proof is constructive.
    This also provides the first nontrivial lower bound for the function $f$ proposed by Erd\H{o}s and Hajnal, should such a function exist.

    We also study this problem in graphs with a forbidden induced subgraph.
    We show that if $G$ is $K_{1,t}$-free, for $t\geq 4$, then $\chi(G) \leq (t-1)(r(G)+\binom{t-1}{2}-3)$.
    If $G$ is claw-free, then we prove $\chi(G) \leq 2r(G)$.
    Additionally, the graphs $G$ where every induced subgraph $G'$ of $G$ satisfy
    $\chi(G') = r(G')$ are considered. 
    We call such graphs path-perfect, as this class generalizes perfect graphs.
    We prove that if $H$ is a forest with at most $4$ vertices other than the claw, then 
    every $H$-free graph $G$ has $\chi(G) \leq r(G)+1$.
    We also prove that if $H$ is additionally not isomorphic to $2K_2$ or $K_2+2K_1$, then all $H$-free graphs are path-perfect.
\end{abstract}

\maketitle

\section{Introduction}

\subsection{Background}

The chromatic number of a graph $G$, denoted $\chi(G)$, is the minimum number of colours
needed to colour the vertices of $G$ such that adjacent vertices always receive distinct colours.
The clique number of a graph $G$, denoted $\omega(G)$, is the size of a maximum clique in $G$.
If $A \subseteq V(G)$, then $G[A]$ denotes the subgraph of $G$ induced by $A$
and if $H$ is a subgraph of $G$, then we write $G[H]$ to mean $G[V(H)]$.
Let $r(G)$ be the maximum chromatic number of any subgraph $G[P]$ of $G$,
where $P$ is a path in the graph $G$.

Obviously, $\omega(G) \leq r(G) \leq \chi(G)$. 
However, it is unclear whether there exists a function $h$ such that $\chi(G) \leq h(r(G))$.
In fact, the authors are unaware of any graphs $G$ appearing in the literature
that have been shown to satisfy $\chi(G) > r(G) +1$.
It is natural to ask, do such graphs exist?
To see that $\chi(G) = r(G) +1$ is achievable for all $r\geq 3$
one can take a sufficiently large $(r+1)$-critical Gallai graph,
defined in \cite{gallai1963kritische}.
Two graphs with $r = 3$ and chromatic number $4$ are
shown in Figure~\ref{fig:Gallai3-4Critical}.
Trivially, $\chi(G) = r(G)$ for all graphs with $r(G) \leq 2$.

\begin{figure}[h!]
\begin{center}
\scalebox{0.9}{
\includegraphics[width=0.4\linewidth]{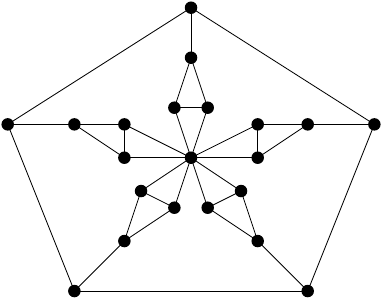}
\hspace{0.25cm}
\includegraphics[width=0.4\linewidth]{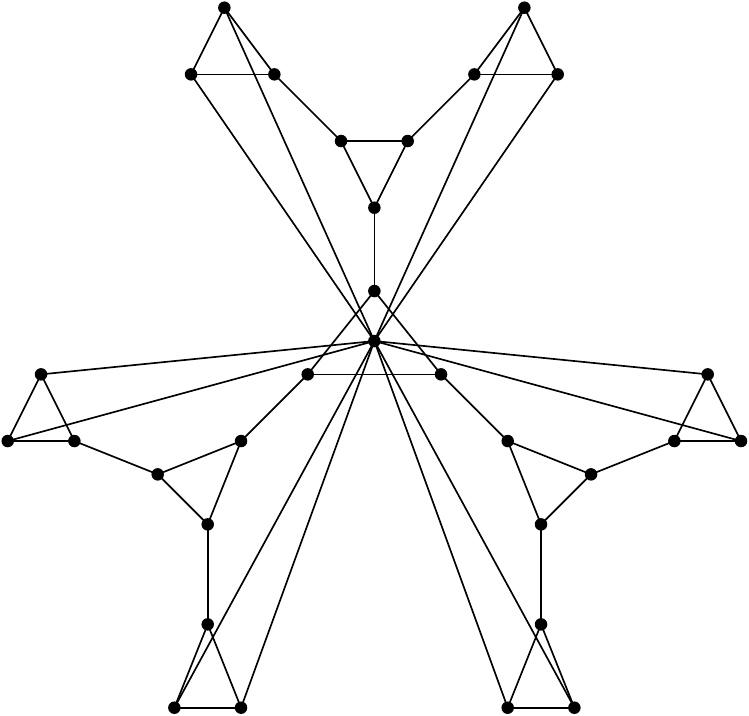}
}
\end{center}
    \caption{Two graphs with $r = 3$ and chromatic number $4$.
    The left graph is taken from Randerath and Schiermeyer \cite{randerath2002chromatic}, 
    while the right graph is a non-traceable $4$-critical graph discovered by Gallai \cite{gallai1963kritische}.}
    \label{fig:Gallai3-4Critical}
\end{figure}

Determining if there exists a function $h$, such that for all graphs $G$, $\chi(G) \leq h(r(G))$
is a relaxed version of the following problem:
Is there a function $f$ such that if each odd cycle of a graph spans a subgraph with chromatic number at most $r$, then the chromatic number the graph is at most $f(r)$?
This problem was published in 1997 by Gy\'{a}rf\'{a}s \cite{gyarfas1997fruit} where it is attributed to Hajnal and Erd\H{o}s.

In fact, Gy\'{a}rf\'{a}s goes further in \cite{gyarfas1997fruit} by making the following conjecture.
The statement has been modified to employ our $r(G)$ notation (see Conjecture~2 in \cite{gyarfas1997fruit} for the original statement).
This conjecture recently resurfaced as Question~1.6 in~\cite{gyarfas2023problems}, a 2023 paper by Gy\'{a}rf\'{a}s .

\begin{conjecture}\label{conjecture: r=3,k=4}
    There exists a constant $k$ (perhaps $k=4$) such that if $G$ is a graph with $r(G)\leq 3$, then $\chi(G) \leq k$. 
\end{conjecture}

The best progress on this conjecture comes from Randerath and Schiermeyer \cite{randerath2002chromatic},
who proved that $\chi(G) \leq r(G) \log_{\frac{8}{7}}(n)$ for all graphs $G$ of order $n$.
The authors are unaware of any previous or subsequent progress on this problem.

We now take a diversion to mention the related problem of $\chi$-boundedness.
Here,
$\chi$-boundedness refers to the study of which
graph classes $\mathcal{G}$ admit functions $f$ such that for all $G\in \mathcal{G}$, 
$\chi(G) \leq f(\omega(G))$.
Motivated by the study of perfect-graphs
(i.e. graphs $G$ where for every induced subgraph $H$ of $G$, $\chi(H) = \omega(H)$)
Gy\'{a}rf\'{a}s began the study of $\chi$-boundedness in \cite{gyarfas1975ramsey}.
Significantly, the famous result by Erd\H{o}s \cite{erdos1959graph}
that there exists graphs with large girth and arbitrarily large chromatic number
implies that many graph classes are not $\chi$-bounded.
In particular, $H$-free graphs are not $\chi$-bounded for any choice of $H$ that contains a cycle.
This leads to the famous Gy\'{a}rf\'{a}s-Sumner conjecture \cite{gyarfas1975ramsey,sumner1981subtrees}
which states that the class of $H$-free graphs are $\chi$-bounded if and only if $H$ is a forest.
For more on $\chi$-boundedness we refer the reader to Scott and Seymour's survey \cite{scott2020survey} and to Char and Karthick's very recent survey for subfamilies of $P_t$-free graphs~\cite{charkarthick2025survey}.

Significantly, one can view the connection between chromatic number and $r$ as a generalization of $\chi$-boundedness.
It is obvious that if a graph class $\mathcal{G}$ is $\chi$-bounded, 
then there exists a function $f$ where for all $G\in \mathcal{G}$, $\chi(G) \leq f(r(G))$.
However, we have no evidence that the converse is true.
In fact, Conjecture~\ref{conjecture: r=3,k=4} is an explicit statement to the contrary.

If one is to believe that there exists a function $h$, such that $\chi(G) \leq h(r(G))$ for all graphs $G$,
then this difference becomes even more extreme.
Of course, the same can be said for considering bounds on chromatic number in terms of the maximum chromatic number of subgraphs spanned by odd cycles, as proposed by Hajnal and Erd\H{o}s.
Our goal here is to take some first steps to exploring this exciting, hard to tackle, area.

\subsection{Our Results}

Our primary contribution is the following theorem.
This result can be viewed as a lower bound on the proposed functions $h$ and $f$, mentioned in Section~1.1,
should either function exist.
Note that this only improves
existing results
for $r\geq 6$.

\begin{theorem}\label{Thm: Lower result}
For all natural numbers $r$, there exists a graph $G$ with $r(G) \leq r$ and 
$$
\chi(G) \geq \Big\lfloor \frac{3r}{2} \Big\rfloor -1.
$$
\end{theorem}

This theorem provides a nice corollary, which answers our question regarding the existence of graphs with $\chi > r+1$.

\begin{corollary}\label{Coro: Arbitrary Gap}
    For all constants $k$, there exists a graph $G$ with 
    $$
    \chi(G) - r(G) > k.
    $$
\end{corollary}

Our proof is constructive and deterministic. 
That is, for all $r$, we demonstrate a graph $G$ with $r(G) = r$ and $\chi(G) = \lfloor \frac{3r}{2}\rfloor -1$.
It remains unclear if there exists a constant $c$ such that $\chi(G)~\leq c\cdot r(G)$ for all graphs.
However, Theorem~\ref{Thm: Lower result} is best possible using our construction.

In addition to this,
and in the vein of $\chi$-boundedness,
we prove upper bounds on the chromatic number of graphs 
with a forbidden induced subgraph.
The first of our results here concerns graphs with a forbidden induced star.
Note that $K_{1,t}$-free graphs were shown to be $\chi$-bounded by Gy\'{a}rf\'{a}s in \cite{gyarfas1987problems},
where near optimal $\chi$-binding functions were determined for all large enough $t$.
When $t$ is large, the best $\chi$-binding functions for $K_{1,t}$-free graphs are far from our functions that bind
the chromatic number of $K_{1,t}$-free graphs in terms of $r(G)$.

\begin{theorem}\label{Thm: K_1,t-free}
Let $t\geq 4$ be an integer.
If $G$ is a $K_{1,t}$-free graph and $r(G)\leq r$, then 
$$
\chi(G) \leq (t-1)\Big(r+\binom{t-1}{2}-3\Big).
$$
If $G$ is $K_{1,3}$-free, then $\chi(G) \leq 2r(G)$.
\end{theorem}

Significantly,
Chudnovsky and Seymour \cite{chudnovsky2010claw} proved that if $G$ is a connected $K_{1,3}$-free graph
with independence number at least $3$, then $\chi(G) \leq 2\omega(G)$.
Furthermore, they showed this bound to be tight.
Hence, for all such graphs, Theorem~\ref{Thm: K_1,t-free} does not improve the best $\chi$-binding function.
As a result, we do not expect this bound to be tight. Interestingly though, if $\alpha(G) \leq 2$ and $G$ is claw-free, then the bound $\chi(G) \leq \omega(G)^2$ cannot be improved to linear.
To see this observe that there exist claw-free graphs $G$ of order $n$ with $\alpha(G) \leq 2$, so $\chi(G) \geq \frac{n}{2}$,  such that $\omega(G) = O(\sqrt{n\log(n)})$ \cite{kim1995ramsey}.
It is worth noting that an immediate corollary of the Chv\'{a}tal-Erd\H{o}s condition for a graph to be Hamiltonian (see Theorem~\ref{Thm: Chvatel-Erdos}), that if $\alpha(G) \leq 2$, then $\chi(G) = r(G)$.

Next, we consider the analogue of perfect graphs for this parameter $r$.
That is, we study which graphs $H$ satisfy that every $H$-free graph $G$ has $\chi(G) = r(G)$.
If $G$ is perfect, then $\chi(G) = r(G)$ trivially, however the converse is not true.
For example, we will prove that every $(P_3+K_1)$-free graph $G$ satisfies  $\chi(G) = r(G)$.
However, one can easily verify that this class is not perfect using the strong perfect graph theorem \cite{chudnovsky2006strong}.
Notice that one can also verify $(P_3+K_1)$-free graphs are not perfect by combining 
the weak perfect graph theorem \cite{lovasz1972normal} and the famous result by Erd\H{o}s \cite{erdos1959graph}
that there exists graphs with high girth and high chromatic number.

This makes the 
analogue of perfect graphs for $r$ interesting.
Formally, 
if every induced subgraph $H$ of a graph $G$ satisfies $\chi(H) = r(H)$, then we say $G$ is \emph{path-perfect}.
We note that proving every graph in a family $\mathcal{G}$ is path-perfect is equivalent to showing that every vertex-critical graph in $\mathcal{G}$ is traceable, as every graph that is not vertex-critical will have the same chromatic number as one of its vertex-critical induced subgraphs.
We prove the following theorem for graphs forbidding a small induced forest.

\begin{theorem}\label{Thm: Small Forest Free}
    If $H$ is a forest on at most $4$ vertices not isomorphic to $K_{1,3}$, then every $H$-free graph $G$ has $\chi(G) \leq r(G)+1$.
    If $H$ is additionally not isomorphic $2K_2$ or $K_2 + 2K_1$, then every $H$-free graph is path-perfect.
\end{theorem}

We believe that all forests $H$ on at most $4$ vertices, baring perhaps the claw ($K_{1,3}$),
every $H$-free graph is path-perfect.
However, we leave proving that $2K_2$-free graphs and $(K_2 + 2K_1)$-free graphs 
are path-perfect as an open problem.
For more on this, and other open problems, see Section~\ref{sec: Future Work}

The rest of paper is organized as follows.
We begin in Section~\ref{sec: star-free} by considering graphs with a forbidden induced star.
Next, in Section~\ref{sec: small H-free} we study graphs with a forbidden induced forest on at most $4$ vertices.
In Section~\ref{sec: construction} for all $r\geq 6$ we construct a sequence of graphs $\{G^{(r)}_k\}$.
For a fixed $r\geq 6$, the graph $G^{(r)}_{\lfloor\frac{r}{2}\rfloor - 1}$ will achieve the bound in Theorem~\ref{Thm: Lower result}.
In Section~\ref{sec: global colour} we will determine the chromatic number $G^{(r)}_{\lfloor\frac{r}{2}\rfloor - 1}$.
While in Section~\ref{sec: path-colour} we determine $r(G^{(r)}_{\lfloor\frac{r}{2}\rfloor - 1})$.

\subsection{Notation and preliminaries}

We refer the readers to~\cite{Wilson} for most standard graph theory notation, but we will now briefly list some definitions and notation to avoid any confusion.
All graphs considered in this paper are undirected, finite, simple, and loopless. 
Even though all graphs are undirected, we will sometimes use an ordered pair to denote an edge.
Every colouring in this paper is assumed to be a proper colouring.
For example, we may write $G-(u,v)$ to denote the graph obtained be removing the edge between $u$ and $v$.
For a graph $G$ we let $\chi(G)$, $\kappa(G)$, $\alpha(G)$, $\omega(G)$, and $\delta(G)$ denote its chromatic number, vertex-connectivity, independence number, clique number, and minimum degree respectively.
We say that a graph $G$ is $H$-free if it does not contain the graph $H$ as an induced subgraph.
For a graph $G$ and two subsets $S$ and $T$ of $V(G)$, we use $E(S,T)$ to denote the set of edges with one endpoint in $S$ and the other in $T$. We let $G[S]$ denote the subgraph of $G$ induced by $S$. We let $N[S]$ denote the \textit{closed neighbourhood} of the set $S$ consisting of all neighbours of $S$ together with the set $S$ and we let $N(S)=N[S]\setminus S$ be the \textit{open neighbourhood}. When $S=\{v\}$, we simply write $N[v]$ and $N(v)$.
We say a graph is \textit{traceable} if it has a Hamiltonian path. 
A graph $G$ is \textit{vertex-critical} if $\chi(G-v)<\chi(G)$ for all $v\in V(G)$. 
The disjoint union of two graphs is denoted $G+H$. 
The the disjoint union of $\ell$ copies of the same graph $G$ is denoted $\ell G$.
For a positive integer $k$, we let $[k]$ denote the set $\{1,2,\dots,k\}$.

We additionally use some notation from order theory.
We assume the reader knows the definition of a partially ordered set (poset).
Given a poset $(P,\preceq)$, 
$x$ is said to \emph{cover} $y$ if $x\prec y$ and for no other $z\in P$ do we have $x\prec z\prec y$.
A \emph{covered chain} $u_1\prec u_2\prec \cdots \prec u_{\ell}$ is a chain such that $u_i$ covers $u_{i+1}$ for each $i\in \{1,\dots\ell-1\}$.
Also, an element $m\in P$ is said to be \textit{minimal} with respect to $\preceq$ if there is no element $p\in P\setminus\{m\}$ such that $p\preceq m$.

\section{$K_{1,t}$-free Graphs}
\label{sec: star-free}

In this section we will prove upper bounds for the chromatic number of $K_{1,t}$-free graphs $G$ in terms of $r(G)$.
Since $K_{1,t}$-free graphs are $\chi$-bounded, as proven by Gy\'{a}rf\'{a}s \cite{gyarfas1987problems}, 
there trivially exists such bounds.
Our focus then is to prove good upper bounds, rather than the existence of upper bounds, for chromatic number. 

Of course, if $G$ is Hamiltonian, then $\chi(G) = r(G)$, since $G$ is a subgraph of itself and $G$ is spanned by a path.
We will use this observation to great effect in this section.
Specifically, we leverage the well known
Chv\'{a}tal-Erd\H{o}s condition for a graph to be Hamiltonian.

\begin{theorem}[Theorem~1 \cite{chvatal1972note}]\label{Thm: Chvatel-Erdos}
    If $G$ is a graph with at least $3$ vertices and $\kappa(G) \geq \alpha(G)$, then $G$ is Hamiltonian.
\end{theorem}

We begin by examining the chromatic number of graphs with bounded $r$ and bounded independence number.

\begin{lemma}\label{Lemma: Ind = 1,2,3}
    If $G$ is a graph with $\alpha(G) \leq 3$, then $G$ is path-perfect.
\end{lemma}

\begin{proof}
    Let $G$ be a graph with $\alpha(G) \leq 3$. If $G$ has at most $2$ vertices the result is trivial.
    Suppose $G$ is a smallest graph where $\alpha(G) \leq 3$ and $\chi(G) > r(G)$.
    Then $G$ has at least $3$ vertices, $G$ is connected, and $G$ does not contain a Hamilton path.

    If $\kappa(G)\geq 3$, then Theorem~\ref{Thm: Chvatel-Erdos} implies $G$ is Hamiltonian, implying $\chi(G) = r(G)$.
    Otherwise, $1\leq \kappa(G)\leq 2$ since $G$ is connected.
    
    Let $S$ be a vertex-cut of order $\kappa(G)$ in $G$.
    Then the vertices of $G$ can be partitioned into non-empty sets $A,B,S$ where $E(A,B) = \emptyset$
    such that each vertex $v \in S$ has neighbours in $A$ and $B$.
    Since $E(A,B) = \emptyset$, it must be the case that $\alpha(G[A])+ \alpha(G[B]) \leq \alpha(G)\leq 3$.
    As $A$ and $B$ are both non-empty, we have $\alpha(G[A]), \alpha(G[B])\geq 1$.
    So without loss of generality suppose that $\alpha(G[A]) = 1$ and $\alpha(G[B]) \leq 2$.

\vspace{0.25cm}
\noindent\underline{Claim:} If $G[B]$ is connected, then there exists a vertex $w \in B \cap N(S)$ and a Hamilton path $P$ in $G[B]$ where $w$ is an endpoint of $P$.
\vspace{0.25cm}   

Suppose $G[B]$ is connected. 
If $\kappa(G[B]) \geq 2$, then Theorem~\ref{Thm: Chvatel-Erdos} implies there is a Hamilton cycle in $G[B]$.
This implies the claim immediately.

Otherwise,  $\kappa(G[B]) = 1$.
In this case there exists a cut vertex $x$ in $G[B]$.
Let $B_1,B_2$ be the connected components of $G[B]-x$.
Then, $\alpha(B_1)+\alpha(B_2) \leq \alpha(G[B])\leq 2$.
Hence, $B_1$ and $B_2$ induce cliques.

If $B_1 \cap N(S) = \emptyset$, then $G$ is the $1$-clique sum of $G[V(G)\setminus B_1]$ and $G[B_1\cup \{x\}]$ at $x$.
Trivially, this contradicts $G$ being a smallest counter-example.
Suppose then that $B_1\cap N(S) \neq \emptyset$.
By the same argument if $N(S) \cap B_1 = N(x) \cap B_1 = \{y\}$
and $|B_1| > 1$,
then $G$ is the $1$-clique sum of $G[(V(G)\setminus B_1) \cup \{y\}]$ and $G[B_1]$ at $y$.
This also contradicts $G$ being a smallest counter-example.
Suppose then that either $N(S) \cap B_1 = N(x) \cap B_1 = \{y\}$ or $|B_1|>1$ is false.

If $|B_1| = 1$ then the vertex $y \in B_1$ must have a neighbour in $S$.
Let $w = y$.
Since $B_2$ induces a clique, there is a path $P'$ which visits exactly the vertices of $B_2$, with an endpoint that is a neighbour of $x$.
Letting $P'$ be such a path, the path $P = P',x,w$ satisfies all requirements.

Suppose then that $|B_1|>1$.
Then $N(S) \cap B_1 = N(x) \cap B_1 = \{y\}$ is false.
Hence, there exists a vertex $w \in B_1 \cap N(S)$ such that $w$ is not the unique neighbour of $x$ in $B_1$.
Let $w$ be such a vertex and let $y$ be a neighbour of $x$ in $B_1$ distinct from $w$.
Since $B_1$ and $B_2$ induce cliques, there are paths $P_1,P_2$ 
such that $P_2$  visits exactly the vertices of $B_2$, with an endpoint that is a neighbour of $x$,
and $P_1$ visits exactly the vertices of $B_1$ and the endpoints of $P_1$ are $w$ and $y$.
Letting $P_1,P_2$ be such paths, the path $P = P_1,x,P_2$ satisfies all requirements.
\hfill $\diamond$
\vspace{0.25cm}

    If $G[B]$ is connected, then let $w$ and $P$ be a vertex and path as described in the claim.
    Since, $\alpha(G[A])$ induces a clique, and since $S$ is a vertex-cut of smallest size
    it is trivial to verify that $P$ can be extended to a Hamilton path.

    Suppose then that $G[B]$ is disconnected.
    Then $G[B]$ has two components $B_1$ and $B_2$ both of which induce cliques, since 
    $\alpha(B_1)+\alpha(B_2) \leq \alpha(G[B])\leq 2$.
    If $|S| = \kappa(G) = 2$ it is trivial that $G$ has a Hamilton path.
    Otherwise, $|S| = \kappa(G) = 1$. In this case let $S = \{v\}$.
    Then, $G$ is the $1$-sum of $G[A\cup\{v\}]$, $G[B_1\cup\{v\}]$, and  $G[B_2\cup\{v\}]$ all at $v$.
    This contradicts the fact $G$ is a smallest counter-example.

    Therefore, we have demonstrated that $G$ being a smallest counter-example leads to a contradiction.
    This completes the proof.
\end{proof}

\begin{lemma}\label{Lemma: Bounded Ind}
    If $G$ is a graph and $t\geq 4$ an integer such that $\alpha(G) < t$, then 
    $$
    \chi(G) \leq r(G) + \binom{t-1}{2} - 3.
    $$
\end{lemma}

\begin{proof}

We proceed by induction on $t$.
Let $f_4(k) = k$
and for all $i\geq 5$ let
$$
f_i(k) = f_{i-1}(k) + (i-2).
$$
We aim to show for all integers $i\geq4$, if $G$ is a graph with $\alpha(G) < i$, then 
$
\chi(G) \leq f_i(r(G)).
$
If $i=4$, then the result follows by Lemma~\ref{Lemma: Ind = 1,2,3}.
Notice that for all $i\geq 5$,
\begin{align*}
    f_i(r(G)) & = f_4(r(G)) + \sum_{j=3}^{i-2}j\\
    & = r(G) + \binom{i-1}{2} - 3
\end{align*}
which is the upper bound on chromatic number we aim to show for graphs with $\alpha(G) < i$.

Suppose $t\geq 5$ and $\chi(H)\leq f_{t-1}(H)$ for all graphs $H$ satisfying $\alpha(H) < t-1$.
Suppose $G$ is a graph where $\alpha(G) < t$.
If $\kappa(G) \geq t-1 \geq \alpha(G)$, then Theorem~\ref{Thm: Chvatel-Erdos} implies $G$ is Hamiltonian, implying $\chi(G) = r(G) < f_t(r(G))$. 
Otherwise, $\kappa(G) \leq t-2$.
In this case $G$ contains a vertex cut $S\subseteq V(G)$ such that $|S| \leq t-2$.
Let $S$ be such a vertex cut.

Then, $G - S$ is disconnected.
Let $G_1,\dots, G_k$ be the connected components of $G$.
Trivially, for each $1\leq i \leq k$, $\alpha(G_i) < \alpha(G) \leq t-1$ and $r(G_i) \leq r(G)$.
By induction on $t$, for each $1\leq i \leq k$, 
$$
\chi(G_i) \leq f_{t-1}(r(G_i)) \leq f_{t-1}(r(G)).
$$
For each $1\leq i \leq k$ let $\phi_i$ be a $f_{t-1}(r(G))$-colouring of $G_i$.

From here we can colour $G$ by using the same $f_{t-1}(r(G))$ colours for all the vertices in components $G_1,\dots, G_k$ according to the colourings $\phi_1,\dots, \phi_k$, and then giving a new colour to each vertex in $S$.
This requires at most $f_{t-1}(r(G))+(t-2)$ colours.
Hence, 
$$
\chi(G) \leq f_{t-1}(r(G))+(t-2) = f_t(r(G))
$$
as desired. This completes the proof.
\end{proof}

We are now prepared to prove the main result of this section.

\begin{proof}[Proof of Theorem~\ref{Thm: K_1,t-free}]

We will prove that for all integers $t\geq 4$, if $G$ is a $K_{1,t}$-free graph, then 
$$
\chi(G) \leq (t-1)\Big(r(G)+\binom{t-1}{2}-3\Big).
$$
Let $t\geq 4$ be fixed but arbitrary and let $G$ be a $K_{1,t}$-free graph. 
Suppose $\chi(G) = k$, it is sufficient to prove the result when $G$ is $k$-vertex-critical.
Suppose then that $G$ is $k$-vertex-critical.

As $G$ is $k$-vertex-critical $\delta(G)\geq k-1$.
Let $v\in V(G)$ be fixed but arbitrary and let $H = G[N[v]]$.
Then $|V(H)|\geq k$ and $\alpha(H)<t$.
Trivially, $r(H) \leq r(G)$.
Hence, Lemma~\ref{Lemma: Bounded Ind} implies that
\begin{align*}
    \chi(H) \leq r(G) + \binom{t-1}{2} - 3.
\end{align*}
Furthermore, applying the standard inequality relating chromatic number, order, and \newline independence number we can observe
\begin{align*}
    \frac{k}{t-1} \leq \frac{|V(H)|}{\alpha(H)} \leq \chi(H)
\end{align*}
implies that
\begin{align*}
    \chi(G) = k \leq (t-1)\Big(r(G)+\binom{t-1}{2}-3\Big).
\end{align*}
Therefore, we have demonstrated $G$ can be coloured with the desired number of colours.
The reader can easily verify the same argument implies that $K_{1,3}$-free graphs $G$ have $\chi(G) \leq 2r(G)$.
\end{proof}

\section{Forbidding Small Induced Forests}
\label{sec: small H-free}

In this section we will prove Theorem~\ref{Thm: Small Forest Free}.
This amounts to proving that for all forests $H$ with at most $4$ vertices, other than the claw,
every $H$-free graph $G$ has $\chi(G) \leq r(G)+1$.
See Figure~\ref{fig:4-vert forests} for a list of all such graphs $H$.
Note that every forest here is a forest of paths.

\begin{figure}[h!]
    \centering
    \includegraphics[scale = 1.05]{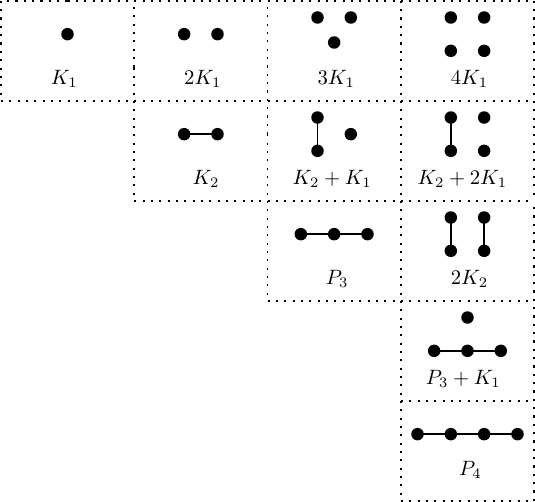}
    \caption{All forests with at most $4$ vertices excluding the $K_{1,3}$.}
    \label{fig:4-vert forests}
\end{figure}

Before proving Theorem~\ref{Thm: Small Forest Free}
we require three lemmas.
The first two lemmas are natural consequences of forbidding a small induced forest.
The third is a non-trivial structural result related to $(P_3+K_1)$-free vertex-critical graphs.

\begin{lemma}\label{Lemma: Longest paths in K_2+2K_1}
    If $G$ is a connected $(K_2+2K_1)$-free and $P$ is a longest path in $G$, then $V(G)\setminus V(P)$ is an independent set.
\end{lemma}

\begin{proof}
    Let $G$ be a $(K_2+2K_1)$-free and $P = v_1,\dots, v_k$ is a longest path in $G$.
    If $P$ is a Hamiltonian path or a path of length $n-1$, then the results is trivial. 
    Suppose then that there exists a vertices $u,w \in V(G)\setminus V(P)$,
    and for contradiction suppose $(u,w) \in E(G)$.
    Since $P$ is a longest path, $v_1$ and $v_k$ are not adjacent to $u$ or $w$.
    
    As $G$ is $(K_2+2K_1)$-free, $G[\{u,w,v_1,v_k\}]$ cannot induce a graph isomorphic to $K_2+2K_1$.
    Hence, $(v_1,v_k) \in E(G)$.
    So the vertices of $V(P)$ are spanned by a cycle.
    Since $G$ is connected, either $u$ or $w$ is adjacent to some vertex on $P$,
    or there exits an internal vertex $z$ on a path $P'$ which connects $V(P)$ and $u$, such that no internal vertex of $P'$ is on $P$.
    But this contradicts the fact that $P$ is longest.
\end{proof}

\begin{lemma}\label{Lemma: Longest paths in 2K_2}
    If $G$ is a $2K_2$-free and $P$ is a longest path in $G$, then $V(G)\setminus V(P)$ is an independent set.
\end{lemma}

\begin{proof}
    Let $G$ be a $2K_2$-free and $P = v_1,\dots, v_k$ is a longest path in $G$.
    If $P$ is a Hamiltonian path or a path of length $n-1$, then the results is trivial. 
    Suppose then that there exists a vertices $u,w \in V(G)\setminus V(P)$,
    and for contradiction suppose $(u,w) \in E(G)$.
    
    Since $G$ is $2K_2$-free, $G[\{u,w,v_1,v_2\}]$ cannot induce a graph isomorphic to $2K_2$.
    Hence, $E(\{u,w\}, \{v_1,v_2\})$ is non-empty.
    Since $P$ is longest, $v_1$ is not adjacent to $u$ or $w$, 
    so suppose without loss of generality that 
    $(u,v_2) \in E(G)$.
    Then $P' = w,u,v_2,v_3,\dots, v_k$ is a longer path than $P$.
    But this is a contradiction.
\end{proof}

\begin{lemma}[Theorem~3.1 \cite{cameron2022dichotomizing}]
\label{Lemma: P_3+K_1 vertex-critical}
    If $G$ is $(P_3+K_1)$-free and $G$ is vertex-critical, then $\alpha(G)\leq 2$.
\end{lemma}

We are now prepared to prove the main result of this section.

\begin{proof}[Proof of Theorem~\ref{Thm: Small Forest Free}]

Let $H$ be a fixed forest with at most $4$ vertices other than $K_{1,3}$.
If $H \in \{K_1,K_2,2K_1, P_3, K_2+K_1, P_4\}$, then every $H$-free graph is perfect.
Hence, for such an $H$, every $H$-free graph $G$ is path-perfect.
If $H \in \{3K_1, 4K_1\}$, then Lemma~\ref{Lemma: Ind = 1,2,3} implies every $H$-free graph is path-perfect.
Otherwise, $H \in \{K_2+2K_1, 2K_2, P_3+K_1\}$.

Consider if $H =  P_3+K_1$. 
As noted above, it suffices to show that every vertex-critical $(P_3+K_1)$-free graph is traceable. 
Since no vertex-critical graph $G$ can have $\kappa(G)=1$ (see, for example \cite[Lemma 2.3]{Dhaliwal2017}), it follows from Theorem~\ref{Thm: Chvatel-Erdos} and Lemma~\ref{Lemma: P_3+K_1 vertex-critical} that  every vertex-critical $(P_3+K_1)$-free graph is Hamiltonian and therefore traceable.

Finally, let $H \in \{K_2+2K_1, 2K_2\}$.
Trivially, it is sufficient to prove the result for connected graphs.
Let $G$ be a connected $H$-free graph and let $P$ be a longest path in $G$.
Then, Lemma~\ref{Lemma: Longest paths in K_2+2K_1} or Lemma~\ref{Lemma: Longest paths in 2K_2} implies that
$V(G)\setminus V(P)$ is an independent set.
By the definition of $r(G)$, $\chi(G[V(P)]) \leq r(G)$.
Trivially, we can extend an $r(G)$-colouring of $G[V(P)]$ to an $r(G)+1$ colouring of $G$ since $V(G)\setminus V(P)$ is an independent set.
Thus, $\chi(G) \leq r(G)+1$ as required.
This concludes the proof.
\end{proof}

\section{Constructing $G^{(r)}_k$}
\label{sec: construction}

In this section 
we construct graphs whose paths induced subgraphs have small chromatic number, but the graph itself has large chromatic number.
In particular,
for all $r\geq 6$
we demonstrate a sequence of graphs $\{G^{(r)}_k\}_{0\leq k \leq \lfloor\frac{r}{2}\rfloor -1}$ 
such that for all valid choices of $r$ and $k$,
$$r(G^{(r)}_k) = r \hspace{1cm} \text{ and} \hspace{1cm} \chi(G^{(r)}_k) = r+k.$$
We prove that the graphs we construct satisfy these claims in later sections.
Notice that by taking $k = \lfloor \frac{r}{2} \rfloor -1$, these properties imply Theorem~\ref{Thm: Lower result}.

This construction is broken into distinct stages to assist the reader.
These stages serve the additional purposes of simplifying later sections of the paper.

\begin{figure}[ht!]
    \centering
\scalebox{1.5}{
\includegraphics[width=0.5\linewidth]{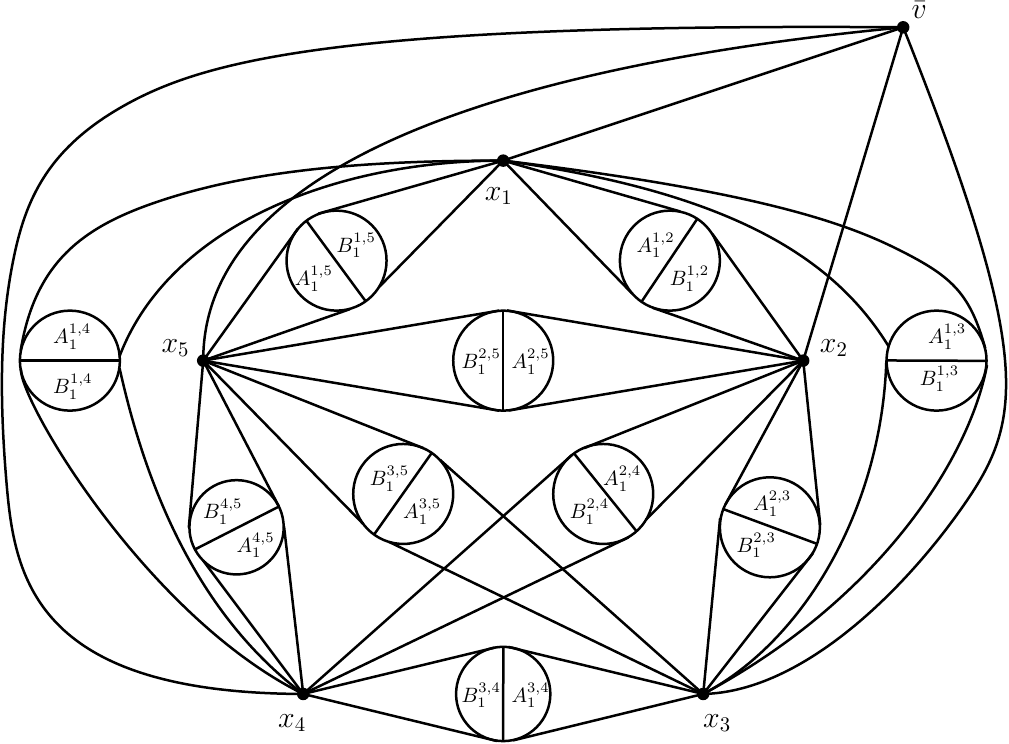}
}
    \caption{The graph $H_1^{(6)}$. Here, $r=6$ and $k=0$, hence there are $5$ copies of $G^{(6)}_0$.}
    \label{fig: Stage1}
\end{figure}

Let $r \geq 6$ be a fixed but arbitrary integer
and let $G^{(r)}_0$ be $K_r$ with vertices labeled $v_1,\dots, v_{r}$.
Let $A^{(r)}_0 = \{v_1,\dots, v_{\lfloor \frac{r}{2}\rfloor}\}$ and $B^{(r)}_0 = \{v_{\lfloor \frac{r}{2}\rfloor+1},\dots, v_{r} \}$.
Given $G^{(r)}_k$, $A^{(r)}_k$, and $B^{(r)}_k$ for $k < \lfloor\frac{r}{2}\rfloor -1$ we define 
$G^{(r)}_{k+1}$, $A^{(r)}_{k+1}$, and $B^{(r)}_{k+1}$ as follows.

\begin{figure}[ht!]
    \centering
\scalebox{1.25}{
\includegraphics[width=0.5\linewidth]{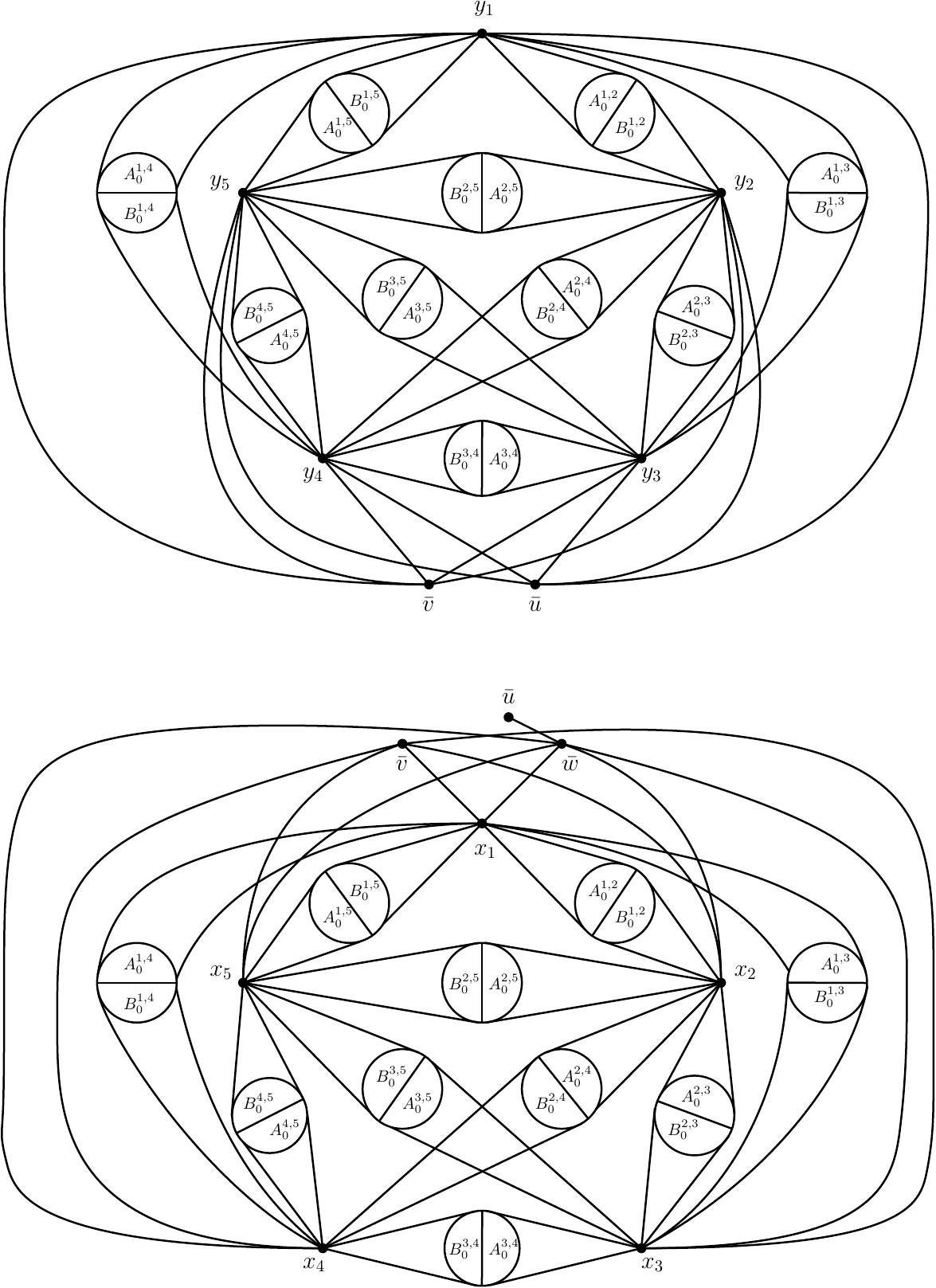}
}
    \caption{The graph $H_1^{(6),+}$ (bottom) and the graph $H_1^{(6),-}$ (top).}
    \label{fig: Stage2}
\end{figure}

The first stage, see Figure~\ref{fig: Stage1}, involves defining a graph $H^{(r)}_{k+1}$ as follows:
Let $G_k^{i,j}$ be copies of $G_k^{(r)}$ indexed by all pairs $1\leq i< j\leq r+k-1$.
It is sometimes convenient to write $G_k^{j,i}$ when $i<j$, 
in this case take $G_k^{j,i}$ to be $G_k^{i,j}$.
For each pair $i<j$ let $A_k^{i,j}$ be the set $A^{(r)}_k$ in $G_k^{i,j}$.
Similarly, for each pair $i<j$ let $B_k^{i,j}$ be the set $B^{(r)}_k$ in $G_k^{i,j}$.
The vertices of $H^{(r)}_{k+1}$ are defined by 
\begin{align*}
    V(H^{(r)}_{k+1}) = \{\bar{v}\} \cup \{x_1,\dots, x_{r+k-1}\} \cup \Bigg(\bigcup_{1\leq i< j\leq k} V(G^{i,j}_{k}) \Bigg).
\end{align*}
For each pair $i<j$, $V(G^{i,j}_{k})$ induces $G_k^{i,j}$, while $\{x_1,\dots, x_{r+k-1}\}$ is an independent set,
and $N(\bar{v}) = \{x_1,\dots, x_{r+k-1}\}$.

All remaining edges are described as follows: 
if $i = 1$ and $j = r+k-1$, then $x_i$ is adjacent to every vertex in $B_k^{i,j}$, while $x_j$ is adjacent to every vertex $A_k^{i,j}$.
For all other pairs $i < j$, $x_i$ is adjacent to every vertex in $A^{i,j}_k$,
while $x_j$ is adjacent to all vertices in $B_k^{i,j}$. 
Here $x_i$ and $x_j$ have no common neighbours in $G_k^{i,j}$, and each vertex $x_q$, where $q\neq i,j$, has no neighbours in $G_k^{i,j}$.

\begin{figure}[h!]
    \centering
\scalebox{1.25}{
\includegraphics[width=0.5\linewidth]{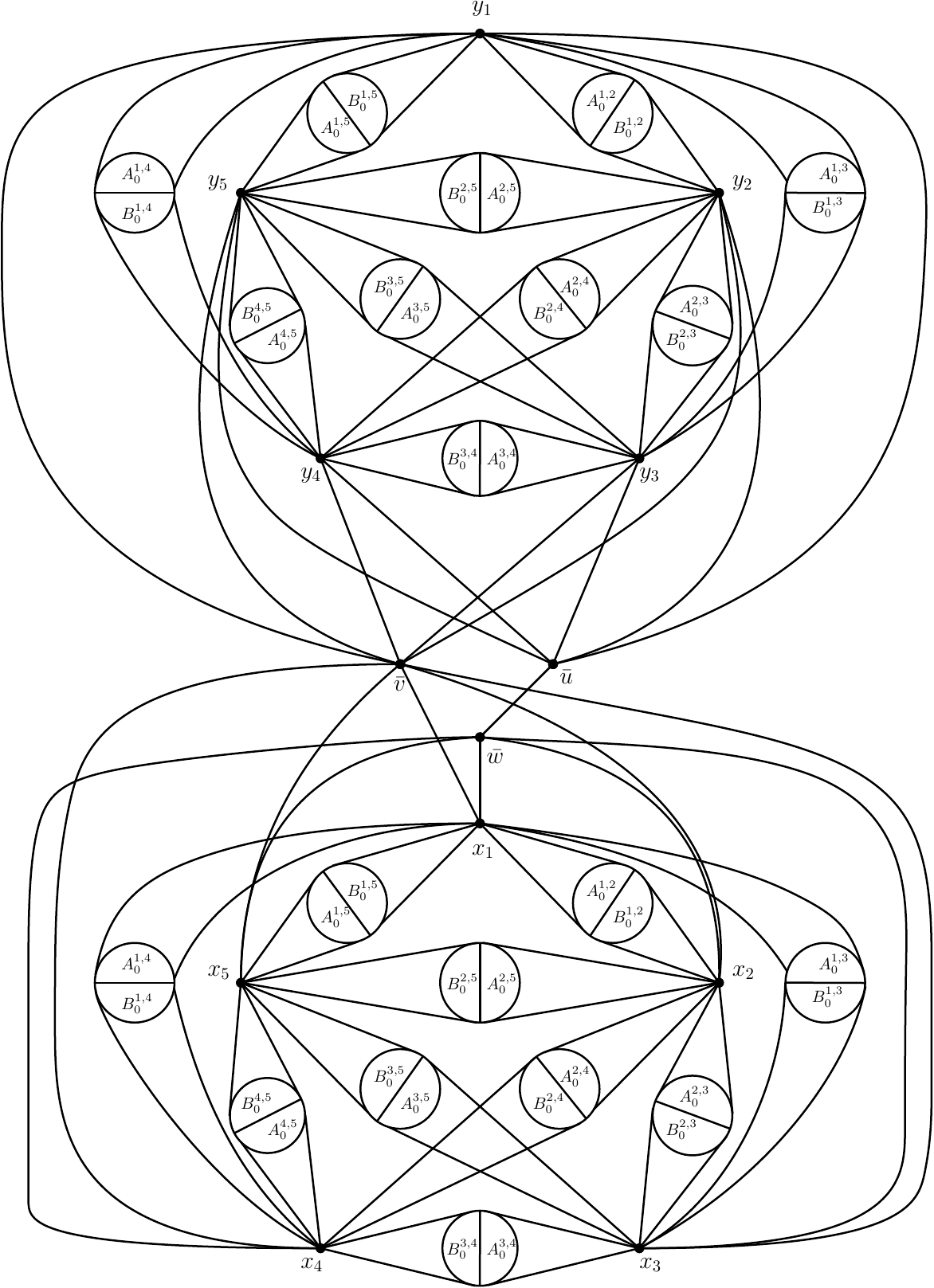}
}
    \caption{The graph $G^{(6)}_1$.}
    \label{fig: Stage3}
\end{figure}

We are now prepared to consider the second stage of our construction. 
Given $H^{(r)}_{k+1}$, we construct two graphs we call $H^{(r),+}_{k+1}$ and $H^{(r),-}_{k+1}$.
To form $H^{(r),+}_{k+1}$ we add two vertices, $\bar{u},\bar{w}$ to $H^{(r)}_{k+1}$.
We let $N(\bar{w}) = N_{H^{(r)}_{k+1}}(\bar{v}) \cup \{\bar{u}\}$, and we let $N(\bar{u}) = \{\bar{w}\}$.
Edges are not added between the vertices of $H^{(r)}_{k+1}$.
Similarly, to form 
$H^{(r),-}_{k+1}$ we add one vertex, $\bar{u}$, to $H^{(r)}_{k+1}$.
We let $N(\bar{u}) = N_{H^{(r)}_{k+1}}(\bar{v})$.
Edges are not added between the vertices of $H^{(r)}_{k+1}$.

The third and final stage of the construction involves identifying vertices in 
$H_{k+1}^{(r),+}$ and $H_{k+1}^{(r),-}$
to form $G^{(r)}_{k+1}$.
Specifically, vertex $\bar{v}$ from $H_{k+1}^{(r),+}$ and vertex $\bar{v}$ from $H_{k+1}^{(r),-}$
are identified,
and vertex $\bar{u}$ from $H_{k+1}^{(r),+}$ and vertex $\bar{u}$ from $H_{k+1}^{(r),-}$,
are identified.
The graph given by identifying these two pairs of vertices is $G^{(r)}_{k+1}$.
To make proofs easier to follow, 
we relabel the vertices $x_1,\dots, x_{r+k-1}$ in $H_{k+1}^{(r),-}$
to $y_1,\dots, y_{r+k-1}$.
See Figure~\ref{fig: Stage1} for a drawing of $G^{(6)}_{1}$.

To complete the construction, we define $A^{(r)}_{k+1}$ and $B^{(r)}_{k+1}$.
To delineate the vertex sets $A^{i,j}_k$, $B^{i,j}_k$ in $H_{k+1}^{(r),+}$ and $H_{k+1}^{(r),-}$
we denote the copies in $H_{k+1}^{(r),+}$ as $A^{i,j}_{k,+}$, $B^{i,j}_{k,+}$,
and we denote the copies in $H_{k+1}^{(r),-}$ as $A^{i,j}_{k,-}$, $B^{i,j}_{k,-}$.
Given this notation let 
$$
A^{(r)}_{k+1} = \{\bar{u},\bar{v}\} \cup \Big(\bigcup_{1\leq i < j \leq r+k-1} A^{i,j}_{k,+} \Big) \cup \Big(\bigcup_{1\leq i < j \leq r+k-1} A^{i,j}_{k,-} \Big).
$$
From here let $B^{(r)}_{k+1} = V(G^{(r)}_{k+1})\setminus A^{(r)}_{k+1}$, that is
$$ 
B^{(r)}_{k+1} = \{x_1,\dots, x_{r+k-1}\}\cup\{\bar{w}\} \cup \Big(\bigcup_{1\leq i < j \leq r+k-1} B^{i,j}_{k,+} \Big) \cup \Big(\bigcup_{1\leq i < j \leq r+k-1} B^{i,j}_{k,-} \Big).
$$

\section{Colouring $G^{(r)}_k$}
\label{sec: global colour}

In this section, we establish $\chi(G^{(r)}_k) = r+k$, as was claimed in Section~\ref{sec: construction}.
Taking $k = \lfloor \frac{r}{2} \rfloor -1$ provides half the proof of Theorem~\ref{Thm: Lower result}.
The other half of the proof, that $r(G^{(r)}_k) = r$ when $k = \lfloor \frac{r}{2} \rfloor -1$, is proven in Section~\ref{sec: path-colour}.

We begin by proving that the graph 
$H^{(r)}_{k+1}$, defined in Section~\ref{sec: construction}, satisfies certain colouring conditions,
provided some technical assumptions regarding $G^{(r)}_k$.
We will go on to show that these colouring properties of $H^{(r)}_{k+1}$
imply that $G^{(r)}_{k+1}$ satisfies the same technical assumptions as $G^{(r)}_k$.

\begin{lemma}\label{Lemma: Colour H_k+1}
    If $\chi(G^{(r)}_{k}) = r + k$ 
    and $G^{(r)}_{k}$ admits a $(r+k)$-colouring $\phi: V(G^{(r)}_{k}) \rightarrow [r+k]$ such that
    $\phi^{-1}(1) \subseteq A^{(r)}_{k}$ and $\phi^{-1}(2) \subseteq B^{(r)}_{k}$,
    then $\chi(H^{(r)}_{k+1}) = r + k$.
    Furthermore, for all $(r+k)$-colourings $\psi:V(H^{(r)}_{k+1}) \rightarrow [r+k]$ of $H^{(r)}_{k+1}$,
    $\psi(x_i) \neq \psi(x_j)$ for all pairs $1\leq i < j \leq r+k-1$.
\end{lemma}

\begin{proof}
    Suppose $\chi(G^{(r)}_{k}) = r + k$ 
    and $G^{(r)}_{k}$ admits a $(r+k)$-colouring $\phi: V(G^{(r)}_{k}) \rightarrow [r+k]$ such that
    $\phi^{-1}(1) \subseteq A^{(r)}_{k}$ and $\phi^{-1}(2) \subseteq B^{(r)}_{k}$.
    As $H^{(r)}_{k+1}$ contains a subgraph isomorphic to $G^{(r)}_{k}$
    $\chi(H^{(r)}_{k+1}) \geq  \chi(G^{(r)}_{k}) = r+k$.
    Hence, to show $\chi(H^{(r)}_{k+1}) = r+k$ it is sufficient to prove $H^{(r)}_{k+1}$ is $(r+k)$-colourable.
    We proceed by constructing an $(r+k)$-colouring of $H^{(r)}_{k+1}$, then we prove that in all $(r+k)$-colourings of  $H^{(r)}_{k+1}$,
    every vertex in the set $\{x_1,\dots, x_{r+k-1}\}$ receives a different colour.

    Let $\psi:V(H^{(r)}_{k+1}) \rightarrow [r+k]$ be defined as follows.
    For all $x_i \in \{x_1,\dots, x_{r+k-1}\}$ let $\psi(x_i) = i$, and let $\psi(\bar{v}) = r+k$.
    For each pair $i,j$ with $1\leq i < j \leq r+k-1$ we will now describe how $\psi$ colours vertices in $G^{i,j}_k$.    
    By assumption there exists an $(r+k)$-colouring $\phi: V(G^{(r)}_{k}) \rightarrow [r+k]$ such that
    $\phi^{-1}(1) \subseteq A^{(r)}_{k}$ and $\phi^{-1}(2) \subseteq B^{(r)}_{k}$.
    Then there exists such an $(r+k)$-colouring $\phi$ of  $G^{i,j}_k$.
    Let $\phi$ be such an $(r+k)$-colouring of $G^{i,j}_k$
    and let $\sigma$ be a permutation of the integers $1,\dots, r+k$ such that $\sigma(1) = j$ and $\sigma(2) = i$ in all cases except if $i=1$ and $j=r+k-1$, where we let $\sigma(2) = r+k-1$ and $\sigma(1) = 1$.
    Let $\phi_{i,j}$ be the $(r+k)$-colouring of $G^{i,j}_k$ given by permuting the colour classes of $\phi$ in  $G^{i,j}_k$ by $\sigma$.
    Then, for all $u \in V(G^{i,j}_k)$ let $\psi(u) = \phi_{i,j}(u)$. 

    We claim $\psi$ is a proper colouring.
     For the sake of contradiction, suppose $\psi$ is not a proper colouring.
    Then there exists adjacent vertices $u,v$ such that $\psi(u) = \psi(v)$.
    Let $u,v$ be such a pair.
    Since every vertex in the set $S = \{\bar{v}, x_1, \dots, x_{r+k-1}\}$ receives a different colour, at most one vertex of the pair $u,v$ is in $S$.
    Furthermore, since $N(\bar{v}) \subseteq S$, neither $u$ nor $v$ is equal to $\bar{v}$.
    Without loss of generality suppose that $u$ is not in $S$. 
    Then there exists a pair $i<j$ such that $u \in V(G^{i,j}_k)$.
    If $u \in V(G^{i,j}_k)$, then every neighbour of $u$ is in $V(G^{i,j}_k) \cup \{x_{i},x_{j}\}$ implying $v \in V(G^{i,j}_k) \cup \{x_{i},x_{j}\}$.
    We cannot have $u,v \in V(G^{i,j}_k)$ since $\psi(u) = \phi_{i,j}(u)$, and $\psi(v) = \phi_{i,j}(v)$, while $\phi_{i,j}$ is a proper colouring of $G^{i,j}_k$.
    
    Suppose then that $v \in \{x_{i},x_{j}\}$.
    First, consider the case where 
    $\{i,j\} \neq \{1,r+k-1\}.$
    Then $v = x_{i}$ if $u \in A^{i,j}_k$ and $v = x_j$ otherwise.
    Again this is a contradiction, 
    since we assume $\psi(u) = \psi(v)$, while if $v = x_i$ and $u \in A^{i,j}_k$, 
    then $\psi(u) \neq i = \psi(x_i)$ given $\psi(u) = \phi_{i,j}(u)$ and $\phi_{i,j}^{-1}(i) = \phi^{-1}(2) \subseteq B^{i,j}_k$. 
    Similarly if $v = x_j$ and $u \in B^{i,j}_k$, 
    then $\psi(u) \neq j = \psi(x_j)$ given $\psi(u) = \phi_{i,j}(u)$ and $\phi_{i,j}^{-1}(j) = \phi^{-1}(1) \subseteq A^{i,j}_k$.
    Thus, $\chi(H^{(r)}_{k+1}) = r+k$ as claimed.
    
    Now, if $i= 1$ and $j=r+k-1$, then the result follows very similarly, except that now by definition, we have that $v = x_{r+k-1}$ if $u \in A^{1,r+k-1}_k$ and $v = x_{1}$ otherwise. 
    Thus, if $v = x_{r+k-1}$ and $u \in A^{1,r+k-1}_k$, 
    then $\psi(u) \neq r+k-1 = \psi(x_{r+k-1})$ given that $\psi(u) = \phi_{1,r+k-1}(u)$ and $\phi_{1,r+k-1}^{-1}(r+k-1) = \phi^{-1}(2) \subseteq B^{1,r+k-1}_k$. 
    Further, if $v = x_{1}$ and $u \in B^{1,r+k-1}_k$, 
    then $\psi(u) \neq 1 = \psi(x_{1})$ given that $\psi(u) = \phi_{1,r+k-1}(u)$ and $\phi_{1,r+k-1}^{-1}(1) = \phi^{-1}(1) \subseteq A^{1,r+k-1}_k$.
    Therefore, $\chi(H_{k}^{(r)}) = r+k$.

    We now prove that if $\psi$ is an $(r+k)$-colouring of $H^{(r)}_{k+1}$, 
    then every vertex in the set $\{x_1,\dots, x_{r+k-1}\}$ receives a different colour.
    Suppose for the sake of contradiction that $\psi$ is an $(r+k)$-colouring of $H^{(r)}_{k+1}$
    such that there exist distinct $i$ and $j$ where $\psi(x_i) = \psi(x_j)$.
    Let $\psi$ be such a colouring and without loss of generality suppose $\psi(x_i) = \psi(x_j) = r+k$.
    By the definition of $H^{(r)}_k$, every vertex in $G^{i,j}_k$ is adjacent to $x_i$ or $x_j$.
    Hence, no vertex in  $G^{i,j}_k$ receives colour $r+k$.
    But this implies $\psi$ restricted to $G^{i,j}_k$ is a $(r+k-1)$-colouring of $G^{i,j}_k$.
    This is a contradiction since $\chi(G^{i,j}_k) = \chi(G^{(r)}_k) > r+k-1$.
    Thus, if $\psi$ is an $(r+k)$-colouring of $H^{(r)}_{k+1}$, 
    then every vertex in the set $\{x_1,\dots, x_{r+k-1}\}$ receives a different colour.
    This completes the proof.
\end{proof}

Next, we consider how the graphs $H^{(r),+}_{k+1}$ and $H^{(r),-}_{k+1}$ can be coloured.
We assume the same technical assumptions regarding $G^{(r)}_k$.

\begin{lemma}\label{Lemma: Colour H_{k+1} minus}
    If $\chi(G^{(r)}_{k}) = r + k$ 
    and $G^{(r)}_{k}$ admits a $(r+k)$-colouring $\phi: V(G^{(r)}_{k}) \rightarrow [r+k]$ such that
    $\phi^{-1}(1) \subseteq A^{(r)}_{k}$ and $\phi^{-1}(2) \subseteq B^{(r)}_{k}$,
    then $\chi(H^{(r),-}_{k+1}) = r + k$.
    Furthermore, for all $(r+k)$-colourings $\psi:V(H^{(r),-}_{k+1}) \rightarrow [r+k]$ of $H^{(r),-}_{k+1}$,
    $\psi(\bar{u}) = \psi(\bar{v})$.
\end{lemma}

\begin{proof}
    Suppose
    $\chi(G^{(r)}_{k}) = r + k$ 
    and $G^{(r)}_{k}$ admits a $(r+k)$-colouring $\phi: V(G^{(r)}_{k}) \rightarrow [r+k]$ such that
    $\phi^{-1}(1) \subseteq A^{(r)}_{k}$ and $\phi^{-1}(2) \subseteq B^{(r)}_{k}$.
    Then Lemma~\ref{Lemma: Colour H_k+1} implies that  $\chi(H^{(r)}_{k+1}) = r + k$ 
    and in any $(r+k)$-colouring of $H^{(r)}_{k+1}$ every vertex in $\{x_1,\dots, x_{r+k-1}\}$ receives a different colour.
    Since $H^{(r)}_{k+1}$ is a subgraph of $H^{(r),-}_{k+1}$, it is immediate that $\chi(H^{(r),-}_{k+1}) \geq \chi(H^{(r)}_{k+1}) = r + k$.
    Any $(r+k)$-colouring $\psi$ of $H^{(r)}_{k+1}$ can be extended to an $(r+k)$-colouring $\psi'$ of $H^{(r),-}_{k+1}$ by letting $\psi'(\bar{u})=\psi(\bar{v})$ and $\psi'(u)=\psi(u)$ for all $u\neq \bar{u}$ since $\bar{u}$ and $\bar{v}$ are non-adjacent and $N(\bar{u})=N(\bar{v})$.
    Thus, $\chi(H^{(r),-}_{k+1}) = r + k$.
    
    Further, every $(r+k)$-colouring of $H^{(r),-}_{k+1}$ necessarily contains an $(r+k)$-colouring of $H_{k+1}^{(r)}$.
    Since $N(\bar{v})=N(\bar{u})=\{x_1,x_2,\dots x_{r+k-1}\}$, with each $x_i$ having different colours in an $(r+k)$-colouring from Lemma~\ref{Lemma: Colour H_k+1}, 
    it is indeed necessary to colour $\bar{u}$ with the same colour as $\bar{v}$.
\end{proof}

\begin{lemma}\label{Lemma: Colour H_{k+1} plus}
    If $\chi(G^{(r)}_{k}) = r + k$ 
    and $G^{(r)}_{k}$ admits a $(r+k)$-colouring $\phi: V(G^{(r)}_{k}) \rightarrow [r+k]$ such that
    $\phi^{-1}(1) \subseteq A^{(r)}_{k}$ and $\phi^{-1}(2) \subseteq B^{(r)}_{k}$,
    then $\chi(H^{(r),+}_{k+1}) = r + k$.
    Furthermore, for all $(r+k)$-colourings $\psi:V(H^{(r),+}_{k+1}) \rightarrow [r+k]$ of $H^{(r),+}_{k+1}$,
    $\psi(\bar{u}) \neq \psi(\bar{v})$.
\end{lemma}

\begin{proof}
    Suppose
    $\chi(G^{(r)}_{k}) = r + k$ 
    and $G^{(r)}_{k}$ admits a $(r+k)$-colouring $\phi: V(G^{(r)}_{k}) \rightarrow [r+k]$ such that
    $\phi^{-1}(1) \subseteq A^{(r)}_{k}$ and $\phi^{-1}(2) \subseteq B^{(r)}_{k}$.
    Since $V(H^{(r),+}_{k+1})\setminus \{\bar{u}\}$ induces a subgraph isomorphic to $H^{(r),-}_{k+1}$, Lemma~\ref{Lemma: Colour H_{k+1} minus} implies that we can colour the graph induced by $V(H^{(r),+}_{k+1})\setminus \{\bar{u}\}$ with $r + k$ colours.
    Since $r\ge 6$, we can clearly extend any $(r+k)$-colouring of $V(H^{(r),+}_{k+1})\setminus \{\bar{u}\}$ to an $(r + k)$-colouring of $H^{(r),+}_{k+1}$ by giving $\bar{u}$ any colour other than the colour used for $\bar{w}$, $\bar{u}$'s only neighbour.
    Thus, $\chi(H^{(r),+}_{k+1})) = r + k$.
    
    Now, for the sake of contradiction, suppose that $H^{(r),+}_{k+1}$ has an $(r+k)$-colouring, $\psi$, such that $\psi(\bar{u})=\psi(\bar{v})$.
    Since any $(r+k)$-colouring of $H_{k+1}^{(r),+}$ necessarily contains an $(r+k)$-colouring of the copy of $H_{k+1}^{(r),-}$ induced by $V(H^{(r),+}_{k+1})\setminus \{\bar{u}\}$, it follows from Lemma~\ref{Lemma: Colour H_{k+1} minus} that $\psi(\bar{v})=\psi(\bar{w})$.
    However, this contradicts the fact that $\psi$ is a proper colouring since $\psi(\bar{u})=\psi(\bar{w})$ but $\bar{u}$ and $\bar{w}$ are adjacent.
    Thus, in every  $(r+k)$-colouring of $H^{(r),+}_{k+1}$, $\bar{u}$ and $\bar{v}$ receive different colours.
\end{proof}

We are now prepared to prove that $G^{(r)}_k$
has chromatic number $r+k$ as we have claimed.
Notice that in the following argument we can let $k$ be arbitrarily large and the claim still holds.
The upper bound on $k$ here is a necessary assumption when proving $r(G^{(r)}_k) = r$ in Section~\ref{sec: path-colour}.

\begin{theorem}
\label{Thm: Colour G^r_k}
    Let $r\geq 6$ and $0 \leq k \leq \lfloor \frac{r}{2} \rfloor - 1$.
    Then, $\chi(G^{(r)}_k) = r+k$.
\end{theorem}

\begin{proof}
    Let $r\geq 6$ be fixed.
    We will prove the following stronger claim.
    For all $r \leq k \leq \lfloor \frac{r}{2} \rfloor - 1$,
    $\chi(G^{(r)}_{k}) = r + k$ 
    and $G^{(r)}_{k}$ admits a $(r+k)$-colouring $\phi: V(G^{(r)}_{k}) \rightarrow [r+k]$ such that
    $\phi^{-1}(1) \subseteq A^{(r)}_{k}$ and $\phi^{-1}(2) \subseteq B^{(r)}_{k}$.
    We induct on $k$.

    For the base case suppose that $k = 0$.
    Then, $G^{(r)}_{k} = K_r$ and the result is trivial.
    Suppose then that $k \geq 0$ and the claim holds for $G^{(r)}_{k}$.
    We will prove the claim holds for $G^{(r)}_{k+1}$.

\vspace{0.25cm}
\noindent\underline{Claim.1:} $\chi(G^{(r)}_{k+1})>r+k$.
\vspace{0.25cm}

    Suppose for the sake of contradiction there exists a $(r+k)$-colouring $\psi$ of $G^{(r)}_{k+1}$.
    Recall that $H_{k+1}^{(r),+}$ and $H_{k+1}^{(r),-}$ are subgraphs of $G^{(r)}_{k+1}$
    such that the vertex sets of these subgraphs intersect in exactly $\{\bar{u},\bar{v}\}$.
    Thus, $\psi$ induces a $(r+k)$-colouring $\psi^+$ of $H_{k+1}^{(r),+}$ 
    and a $(r+k)$-colouring $\psi^-$ of $H_{k+1}^{(r),-}$.
    By the induction hypothesis on $G^{(r)}_{k}$, we can assume the conclusions of Lemma~\ref{Lemma: Colour H_{k+1} minus} and Lemma~\ref{Lemma: Colour H_{k+1} plus}.
    But this implies 
    $$
    \psi^-(\bar{u}) = \psi^-(\bar{v}) = \psi^+(\bar{v}) \neq \psi^+(\bar{u}) = \psi^-(\bar{u})
    $$
    which is a contradiction. 
    Hence, $G^{(r)}_{k+1}$ is not $(r+k)$-colourable.
    Thus, $\chi(G^{(r)}_{k+1})>r+k$ concluding the proof of this claim.
    \hfill $\diamond$

\vspace{0.25cm}
\noindent\underline{Claim.2:} $G^{(r)}_{k+1}$ admits a $(r+k+1)$-colouring $\phi: V(G^{(r)}_{k+1}) \rightarrow [r+k+1]$ such that
    $\phi^{-1}(1) \subseteq A^{(r)}_{k+1}$ and $\phi^{-1}(2) \subseteq B^{(r)}_{k+1}$.
\vspace{0.25cm}

    Recall that $H_{k+1}^{(r),+}$ and $H_{k+1}^{(r),-}$ are subgraphs of $G^{(r)}_{k+1}$
    such that the vertex sets of these subgraphs intersect in exactly $\{\bar{u},\bar{v}\}$.
    By the induction hypothesis on $G^{(r)}_{k}$, 
    we can assume the conclusions of Lemma~\ref{Lemma: Colour H_{k+1} minus} and Lemma~\ref{Lemma: Colour H_{k+1} plus}.
    Hence, 
    $$
    \chi(H_{k+1}^{(r),+}) = \chi(H_{k+1}^{(r),-}) = r+k.
    $$
    Let $\psi^+$ be a $(r+k)$-colouring of $H_{k+1}^{(r),+}$ 
    and let $\psi^-$ be a $(r+k)$-colouring of $H_{k+1}^{(r),-}$.
    Without loss of generality we suppose that $\psi^+(\bar{v}) = \psi^-(\bar{v})$.

    We define an $(r+k+1)$-colouring $\phi$ of $G^{(r)}_{k+1}$ as follows.
    Let $\phi(\bar{v}) = 1$ 
    and let $\phi(\bar{u}) = 1$.
    For every vertex $v \in V(H_{k+1}^{(r),+})\setminus \{\bar{u},\bar{v}\}$
    if $\psi^+(v) \neq 1$, let $\phi(v) = \psi^+(v)+1$.
    If $\psi^+(v) = 1$, then we let 
    $\phi(v) = 1$ if $v \in A^{(r)}_{k+1}$ and 
    $\phi(v) = 2$ if $v \in B^{(r)}_{k+1}$.
    Similarly, for all vertices $w \in  V(H_{k+1}^{(r),-})\setminus \{\bar{u},\bar{v}\}$
     if $\psi^-(w) \neq 1$, let $\phi(w) = \psi^-(w)+1$.
    If $\psi^-(w) = 1$, then we let 
    $\phi(w) = 1$ if $w \in A^{(r)}_{k+1}$ and 
    $\phi(w) = 2$ if $v \in B^{(r)}_{k+1}$.

    We claim $\phi$ is an $(r+k+1)$-colouring of $G^{(r)}_{k+1}$ such that 
    $\phi^{-1}(1) \subseteq A^{(r)}_{k+1}$ and $\phi^{-1}(2) \subseteq B^{(r)}_{k+1}$.
    Suppose not for the sake of contradiction.
    By the definition of $\phi$, it is trivial to verify that
    $\phi^{-1}(1) \subseteq A^{(r)}_{k+1}$ and $\phi^{-1}(2) \subseteq B^{(r)}_{k+1}$,
    so it remains only to show that $\phi$ is a proper colouring. If not, then there exist adjacent vertices $u,v$ in $G^{(r)}_{k+1}$ such that $\phi(u) = \phi(v)$.
    Since, $\{\bar{u},\bar{v}\}$ is a vertex cut that separates the vertices of 
    $V(H_{k+1}^{(r),+})\setminus \{\bar{u},\bar{v}\}$
    from the vertices of $V(H_{k+1}^{(r),-})\setminus \{\bar{u},\bar{v}\}$,
    the vertices $u$ and $v$ must be both in $H_{k+1}^{(r),+}$ or both in $H_{k+1}^{(r),-}$.

    Without loss of generality suppose $u$ and $v$ are both in $H_{k+1}^{(r),+}$.
    Since $u$ and $v$ are adjacent, $\{u,v\}\neq \{\bar{u},\bar{v}\}$.
    Additionally,
    $\psi^+(u) \neq \psi^+(v)$, given that $\psi^+$ is a proper colouring $H_{k+1}^{(r),+}$.
    If $\psi^+(u) \neq 1$ and $\psi^+(v) \neq 1$ this is a contradiction, since we would then have
    $$
    \phi(v) = \phi(u) = \psi^+(u) + 1 \neq \psi^+(v) + 1 = \phi(v),
    $$
    or if $\{u,v\}\cap \{\bar{u},\bar{v}\} \neq \emptyset$, one of $u$ and $v$ receive colour $1$ while the other does not.
    Hence, at least one of $u$ and $v$ receives colour $1$ from $\psi^+$.
    Suppose exactly one of $u$ and $v$ receives colour $1$ from $\psi^+$. 
    Say $v$ without loss of generality. 
    Then $\phi(v) \in \{1,2\}$, while $\phi(u) = \psi^+(u)+1 \geq 3$, unless $u = \bar{u}$ or  $u = \bar{v}$.
    Thus, we suppose $u = \bar{u}$ or  $u = \bar{v}$.
    Then $\phi(u) = 1$, and $v$ is a vertex in $B^{(r)}_{k+1}$, since $u$ and $v$ are adjacent.
    By the definition of $\phi$,  and since $v \in B^{(r)}_{k+1}$ while $\psi^+(v) = 1$, we must have $\phi(v) = 2$, contradicting $\phi(u) = \phi(v) = 1$.
    Suppose then that $\psi^+(u) = \psi^+(v) = 1$.
    This contradicts $u$ and $v$ being adjacent, since $\psi^+$ is a proper colouring of $H_{k+1}^{(r),+}$.
    Therefore, we have proven the claim.
    \hfill $\diamond$

\vspace{0.25cm}
As we have proven the induction statement this concludes the proof.
\end{proof}

\section{Colouring Traceable Subgraphs of $G^{(r)}_k$}
\label{sec: path-colour}

Now we turn our attention to colouring subgraphs of $G^{(r)}_k$ spanned by paths.
To start we prove two useful facts about sparse graphs.
We note that this bound on chromatic number is tight when $H$ is isomorphic to $K_4$.

\begin{lemma}\label{Lemma: Almost forest 4-colour}
    If $H$ is a graph and $T\subseteq E(H)$ such that $|T|\le 5$ and $H-T$ is a forest, then $\chi(H) \leq 4$.
\end{lemma}

\begin{proof}
    Let $H$ be a smallest graph with the property that there exists $T\subseteq E(H)$ with $|T|\le 5$ such that $H-T$ is a forest and $\chi(H) > 4$. 
    Then,  clearly we must have $|V(H)|\geq 5$.
    We claim that $\delta(H) \ge 4$. 
    Suppose not. Then there exists $v\in V(H)$ such that $\deg(v)<4$. 
    We must then have $\chi(G-v) > 4$ as, otherwise, any $4$-colouring of $H-v$ could be extended to a $4$-colouring of $H$ using any of the colours not appearing in the neighbourhood of $v$.
    Further, letting $T'=T\cap E(H-v)$  we have that $(H-v)-T'$ is a forest, since $H-T$ is a forest.
    This contradicts $H$ being the smallest graph with this property.
    Thus, $\delta(H) \geq 4$.
    Then, by the handshaking lemma, $|E(H)| \geq  2|V(H)|$.
    Since $H- {T}$ is a forest, 
    $$|E(H- {T})| \leq|V(H- {T})| -1  = |V(H)| -1.$$
    Hence, since $|V(H)|\geq 5$, we have shown,
    \begin{align*}
        |E(H)| - |E(H- {T})| & \geq 2|V(H)| - |V(H- {T})| +1\\
        & = |V(H)| +1
        \\ &\geq 6
        \\ & > 5
        \\ &  {\ge} |E(H)| - |E(H- {T})|
    \end{align*}
    which is a contradiction.
    Therefore, we conclude $\chi(H) \leq 4$ as desired.
\end{proof}

\begin{lemma}\label{Lemma: Worst case Almost forest 4-colour}
    If $H$ is a graph and  {$T\subseteq E(H)$ with $|T|\le 6$} such that  $H- {T}$ is a forest with at least $2$ connected components, 
    then $\chi(H) \leq 4$.
\end{lemma}

\begin{proof}
    We apply the same argument as Lemma~\ref{Lemma: Almost forest 4-colour}, except $|E(H -  {T})| \leq |V(H)| - 2$, 
    since $H- {T}$ is a disconnected forest.
    This leads to
    \begin{align*}
        |E(H)| - |E(H- {T})| & \geq 2|V(H)| - |V(H- {T})| +2\\
        & = |V(H)| +2
        \\ &\geq 7
        \\ & > 6
        \\ &  {\ge} |E(H)| - |E(H- {T})|
    \end{align*}
   which is a contradiction.
\end{proof}

Next, 
we demonstrate some helpful structure in $G^{(r)}_k$.
In particular, we consider the structure of the induced subgraphs $G^{(r)}_k[A^{(r)}_k]$, and $G^{(r)}_k[B^{(r)}_k]$, as well as the structure of the edge set $E\left(A^{(r)}_k, B^{(r)}_k\right)$.
These observations are not hard to make, 
but they are vital to studying some of the traceable subgraphs of $G^{(r)}_k$.

For this, we introduce some extra notation. 
In $G_{k+1}^{(r)}$ and for each pair $i,j$ with $1\leq i< j\leq r+k-1$, 
let $G^{i,j}_{k,+}$ be the copy of $G_k^{i,j}$ contained in the induced copy of $H_{k+1}^{(r),+}$, 
and $G^{i,j}_{k,-}$ be the copy of $G_k^{i,j}$ contained in the induced copy of $H_{k+1}^{(r),-}$.
We use $G^{i,j}_{k,\pm}$ to refer to $G^{i,j}_{k,+}$ and/or $G^{i,j}_{k,-}$.

\begin{lemma}\label{Lemma: Induced As}
     Let $r\geq 6$ and $0 \leq k \leq \lfloor \frac{r}{2} \rfloor - 1$.
     Every connected component of $G^{(r)}_k[A^{(r)}_k]$ is a clique of size $1$ or $\lfloor \frac{r}{2} \rfloor$.
\end{lemma}

\begin{proof}
    Let $r\geq 6$ be fixed. The claim is trivial for $G^{(r)}_0 \cong K_r$.
    Suppose the claim is true for $k\geq 0$.
    The reader can easily verify every connected component of $G^{(r)}_{k+1}[A^{(r)}_{k+1}]$ is a subgraph of some $G^{i,j}_{k,\pm}$, or is $\{\bar{u}\}$, or is $\{\bar{v}\}$.
    The results follows immediately by induction.
\end{proof}

\begin{lemma}\label{Lemma: Bs next to As}
     Let $r\geq 6$ and $0 \leq k \leq \lfloor \frac{r}{2} \rfloor - 1$.
     If $C$ is a connected component of $G^{(r)}_k[A^{(r)}_k]$
     and $v$ is a vertex such that $N(v) \cap V(C) \neq \emptyset$,
     then $V(C) \subseteq N[v]$.
\end{lemma}

\begin{proof}
    Let $r\geq 6$ be fixed. The claim is trivial for $G^{(r)}_0 \cong K_r$.
    Suppose the claim is true for some $k\geq 0$.
    Let $C$ is a connected component of $G^{(r)}_{k+1}[A^{(r)}_{k+1}]$. 
    If $|C| = 1$, the result it trivial.
    Otherwise, Lemma~\ref{Lemma: Induced As} implies $|C| = \lfloor \frac{r}{2} \rfloor$.
    Moreover, the proof of Lemma~\ref{Lemma: Induced As} implies $C$ is a subgraph of some $G^{i,j}_{k,\pm}$.
    By induction, the result holds for every neighbour of $C$ in $G^{i,j}_{k,\pm}$.
    Thus, let $v\in V(G^{(r)}_{k+1})\setminus V(G^{i,j}_{k,\pm})$ be a vertex such that $N(v) \cap V(C) \neq \emptyset$.
    By definition, $v=x_i$ ($x_j$ if $i=1$ and $j=r+k-1$) if $C\in G^{i,j}_{k,+}$, and, similarly, $v=y_i$ ($y_j$ if $i=1$ and $j=r+k-1$) if $C\in G^{i,j}_{k,-}$.
    Again by definition, $x_i$ ($x_j$ if $i=1$ and $j=r+k-1$) is complete to $A^{i,j}_k\supseteq C$ in $G^{i,j}_{k,+}$ and $y_i$ ($y_j$ if $i=1$ and $j=r+k-1$) is complete to $A^{i,j}_k\supseteq C$ in $G^{i,j}_{k,-}$.
    The result follows by induction.
\end{proof}

\begin{lemma}\label{Lemma: k+2 Bs next to a common A vertex}
    Let $r\geq 6$ and $0 \leq k \leq \lfloor \frac{r}{2} \rfloor - 1$.   
    If $C$ is a connected component of $G^{(r)}_k[A^{(r)}_k]$ not isomorphic to $K_1$ and $b_1,\dots, b_{k+2} \in B^{(r)}_k$ are distinct vertices such that for all $1\leq i \leq k+2$, $N(b_i)\cap V(C)\neq \emptyset$,
    then there exist distinct vertices $b_i,b_j \in \{b_1,\dots, b_{k+2}\}$
    such that
    $
    N[b_i] = N[b_j]
    $
    and $N(b_i) \cap A^{(r)}_k = V(C)$.
\end{lemma}

\begin{proof}
    Let $r\geq 6$ be fixed. The claim is trivial for $G^{(r)}_0 \cong K_r$.
    Suppose the claim is true for some $k\geq 0$.
    Let $C$ be a connected component of $G^{(r)}_{k+1}[A^{(r)}_{k+1}]$. 
    As in the proof of Lemma~\ref{Lemma: Bs next to As}, since $C$ is not isomorphic to $K_1$,
    $C$ is a subgraph of some $G^{i,j}_{k,\pm}$.
    This implies every neighbour of $C$ is in $G^{i,j}_{k,\pm}$, except one vertex, $x_i$  {if $G^{i,j}_{k,+}$ and $i\neq 1$ or $j\neq r+k-1$, otherwise $x_j$ (similarly $y_i$ or $y_j$ if $G^{i,j}_{k,-}$)}.
    Let $b_1,\dots, b_{(k+1)+2} \in B^{(r)}_{k+1}$ be distinct neighbours of $C$.
    At most one of these vertices is $x_i$  {(or similarly $x_j$, $y_i$, or $x_j$)}.
    Hence, at least $k+2$ of them must be in $G^{i,j}_{k,\pm}$.
    The result follows by induction.
\end{proof}

When considering the structure $G^{(r)}_k[B^{(r)}_k]$,
it is useful to define the following set $W^{(r)}_k \subseteq B^{(r)}_k$.
Let $W^{(r)}_0 = \emptyset$.
Suppose $k\geq 0$, and $W^{(r)}_k$ is defined, then we define $W^{(r)}_{k+1}$ 
as the union of all sets $W^{(r)}_{k}$ in components $G^{i,j}_{k,\pm}$, and the set $\{\bar{w}\}$.

We note that $W^{(r)}_k$ is well structured.

\begin{lemma}\label{Lemma: W structure}
    Let $r\geq 6$ and $0 \leq k \leq \lfloor \frac{r}{2} \rfloor - 1$.   
    The set $W^{(r)}_k$ is an independent set and $N(W^{(r)}_k) \cap A^{(r)}_k$ is an independent set.
\end{lemma}

\begin{proof}
    From the definition of $W^{(r)}_k$, it is immediate that $W^{(r)}_k$ is an independent set.
    From the definition of $G^{(r)}_{k}$, $N(\bar{w}) \cap A^{(r)}_k = \{\bar{u}\}$.
    So from the definition of $W^{(r)}_k$, one can easily verify $N(W^{(r)}_k) \cap A^{(r)}_k$ is also an independent set.
\end{proof}

Now we show that the set $B^{(r)}_k\setminus W^{(r)}_k$ is well structured.
In the following lemma we will introduce a partial order 
that will allow us to describe in which step of the construction of $G_{k}^{(r)}$ that vertices in $B_k^{(r)}\setminus W_k^{(r)}$ were added.
Here only the minimal elements in this partial order will be from copies of $G_0^{(r)}$.

\begin{lemma}\label{Lemma: B minus W structure}
    Let $r\geq 6$ and $0 \leq k \leq \lfloor \frac{r}{2} \rfloor - 1$.   
    The set $B^{(r)}_k \setminus W^{(r)}_k$ admits a partial ordering $\preceq$ such that 
    for all distinct vertices $u,v \in B^{(r)}_k \setminus W^{(r)}_k$
    \begin{enumerate}
        \item $u$ and $v$ are adjacent if and only if $u\preceq v$ or $v\preceq u$, and \label{Lemma: B minus W structure (1)}
        \item all maximal chains in the partial order $\preceq$ is length $\lceil \frac{r}{2}\rceil + k$, and \label{Lemma: B minus W structure (2)}
        \item if $u_1 \prec u_2 \prec \dots \prec u_{\lceil\frac{r}{2} \rceil}$ is a covered chain and $u_1$ is minimal, 
        then for all $1 \leq i\leq \lceil\frac{r}{2}\rceil$, $N[u_i] = N[u_1]$, and \label{Lemma: B minus W structure (3)}
        \item if $u$ is minimal, then $N(u) \cap W^{(r)}_k = \emptyset$, and \label{Lemma: B minus W structure (4)}
        \item if $u$ minimal, then $N(u) \cap A^{(r)}_k$ induces a graph isomorphic to a clique of size $\lfloor\frac{r}{2}\rfloor$. \label{Lemma: B minus W structure (5)}
    \end{enumerate}
\end{lemma}

\begin{proof}
    Let $r\geq 6$ be fixed. 
    The claim is trivial for $G^{(r)}_0 \cong K_r$
    since any total ordering of $B^{(r)}_0 = B^{(r)}_0\setminus W^{(r)}_0$ satisfies conditions (\ref{Lemma: B minus W structure (1)})-(\ref{Lemma: B minus W structure (5)}).
    Suppose then that for $k\geq 0$ there exists a partial ordering $\preceq_k$ of $B^{(r)}_k\setminus W^{(r)}_k$ as required.
    We show how to extend $\preceq_k$ to a partial ordering $\preceq_{k+1}$ of $B^{(r)}_{k+1}\setminus W^{(r)}_{k+1}$ 
    such that conditions (\ref{Lemma: B minus W structure (1)})-(\ref{Lemma: B minus W structure (5)}) remain satisfied.

    The vertices of $B^{(r)}_{k+1}\setminus W^{(r)}_{k+1}$
    can be partitioned into copies of $B^{(r)}_{k}\setminus W^{(r)}_{k}$, corresponding to subgraphs $G^{i,j}_{k,\pm}$,
    and the sets $\{x_1,\dots, x_{r+k-1}\}$ and $\{y_1,\dots,y_{r+k-1}\}$.
    Let $B^+$ denote the union of copies of $B^{(r)}_{k}\setminus W^{(r)}_{k}$, corresponding to subgraphs $G^{i,j}_{k,+}$, 
    and the set $\{x_1,\dots, x_{r+k-1}\}$. 
    Similarly, let $B^-$ denote the union of copies of $B^{(r)}_{k}\setminus W^{(r)}_{k}$, corresponding to subgraphs $G^{i,j}_{k,-}$, 
    and the set $\{y_1,\dots, y_{r+k-1}\}$. 
    We define $\preceq_{k+1}$ so that no element of $B^+$ is comparable to an element of $B^-$.
    
    We consider how $\preceq_{k+1}$ compares elements of $B^+$ and note that how $\preceq_{k+1}$ compares elements of $B^-$ is defined analogously.
    Let $u,v \in B^+$.
    If $u$ and $v$ belong to the same subgraph $G^{i,j}_{k,+}$, then $u \preceq_{k+1} v$ if and only if 
    $u \preceq_k v$.
    This is well defined since $G^{i,j}_{k,+}$ is isomorphic to $G^{(r)}_k$.
    If $u$ and $v$ belong to different subgraphs $G^{i,j}_{k,+}$ and $G^{t,\ell}_{k,+}$, 
    then we let $u$ and $v$ be incomparable under $\preceq_{k+1}$.
    Similarly, if $u,v \in \{x_1,\dots, x_{r+k-1}\}$, then we let $u$ and $v$ be incomparable under $\preceq_{k+1}$.
    Finally, if $u$ is in a subgraph $G^{i,j}_{k,+}$ and $v \in  \{x_1,\dots, x_{r+k-1}\}$, 
    then $v \not\preceq_{k+1} u$, and,  {for $i\neq 1$ and $j\neq r+k-1$}, $u \preceq_{k+1} v$ if and only if $v = x_j$ 
     {and,  {for $i= 1$ and $j= r+k-1$}, $u \preceq_{k+1} v$ if and only if $v = x_i$.} 
     {Recall that for $i\neq 1$ and $j\neq r+k-1$}, $x_j$ is the vertex in the pair $\{x_i,x_j\}$
    such that $N(x_j) \cap V(G^{i,j}_{k,+})$ is the set of all vertices in $B^{(r)}_k$ in  $G^{i,j}_{k,+}$ {, and that this vertex is $x_i$ when $i= 1$ and $j= r+k-1$}.

    Given the definition of $G^{(r)}_{k+1}$ and the definition of $\preceq_{k+1}$, the induction hypothesis implies
    condition (\ref{Lemma: B minus W structure (1)}) and condition (\ref{Lemma: B minus W structure (2)}) are immediately satisfied.
     By induction, the only minimal elements must belong to a copy of $G_0^{(r)}$, so (\ref{Lemma: B minus W structure (4)}) and (\ref{Lemma: B minus W structure (5)}) follow immediately.
    Also, by this and the definition of $\preceq_{k+1}$, every covered chain containing a minimal element, must be a subset of $B_0^{(r)}$ of a copy of $G_0^{(r)}$, so (\ref{Lemma: B minus W structure (3)}) follows immediately by the definition of $G_0^{(r)}$.  
    This concludes the proof.
\end{proof}

The final lemma we prove is the core of our induction step in Theorem~\ref{Thm: Technical version of lower result}.
We separate this as a lemma
because the proof is technical.
This improves the readability of the proof for Theorem~\ref{Thm: Technical version of lower result},
which is itself non-trivial.

Prior to stating the lemma, we must define the following auxiliary graph.
Given $r\geq 6$ and $0 \leq k \leq \lfloor \frac{r}{2} \rfloor - 1$,
let $\mathcal{G}^{(r)}_k$  be the graph obtained from $G^{(r)}_k$
by adding vertices $v^A,v^B$ such that
$$
N(v^A) = A_{k}^{(r)} \cup \{v^B\} \hspace{1cm}\text{and} \hspace{1cm}  N(v^B) = B_{k}^{(r)} \cup \{v^A\}.
$$
These extra vertices are useful when considering how paths in $G_{k+1}^{(r)}$ intersect subgraphs isomorphic to $G_{k}^{(r)}$.

\begin{lemma}\label{Lemma: Induction-colour-component}
Let $r\geq 6$ and $0 \leq k < \lfloor \frac{r}{2} \rfloor - 1$ and
let $\preceq$ be a partial order of $B^{(r)}_{\lfloor\frac{r}{2} \rfloor - 1} \setminus W^{(r)}_{\lfloor\frac{r}{2} \rfloor - 1}$
as defined in Lemma~\ref{Lemma: B minus W structure}.
If for all paths $P'$ in $\mathcal{G}^{(r)}_k$, and for all distinct colours $a',b' \in \{r-1,r\}$ when $k = 0$,
and all distinct colours $a',b' \in \{r-3,r-2,r-1,r\}$ when $k\geq 1$,
there is an $r$-colouring $\phi': V(P') \rightarrow [r]$ of $\mathcal{G}^{(r)}_k[V(P')]$
such that 
\begin{enumerate}
    \item if $k = 0$,  then for all $1\leq \ell \leq \lfloor \frac{r}{2} \rfloor - 1$, $(\phi')^{-1}(\ell) \subseteq A^{(r)}_0$, and \label{Lemma: Induction-colour-component (1)}
    \item if $k = 0$, then for all $\lfloor \frac{r}{2} \rfloor - 2\leq \ell \leq r - 2$, $(\phi')^{-1}(\ell) \subseteq B^{(r)}_0$, and \label{Lemma: Induction-colour-component (2)}
    \item if $k \geq 1$, then for all $1\leq \ell \leq \lfloor \frac{r}{2} \rfloor - k - 2$, $(\phi')^{-1}(\ell) \subseteq A^{(r)}_k$, and \label{Lemma: Induction-colour-component (3)}
    \item if $k \geq 1$, then for all $\lfloor \frac{r}{2} \rfloor - k - 1\leq \ell \leq r - 2k - 4$, $(\phi')^{-1}(\ell) \subseteq B^{(r)}_k$, and \label{Lemma: Induction-colour-component (4)}
    \item if $v^A$ is on the path $P'$, then $\phi(v^A) = a'$, and \label{Lemma: Induction-colour-component (5)}
    \item if $v^B$ is on the path $P'$, then $\phi(v^B) = b'$, and \label{Lemma: Induction-colour-component (6)}
    \item if $k \leq \lfloor \frac{r}{2} \rfloor -3$, then
        the vertices of 
        $V(P') \cap \left( W^{(r)}_{k} \cup \{v: v \text{ is minimal in } \preceq\}\right)$ are a colour class in $\phi'$,
        otherwise $k = \lfloor \frac{r}{2} \rfloor -2$ and the vertices of 
        \[
        \{v^A\} \cup \Big(V(P') \cap \left( W^{(r)}_{k} \cup \{v: v \text{ is minimal in } \preceq\}\right)\Big)
        \]
        (or $y_i$) are a colour class in $\phi'$. \label{Lemma: Induction-colour-component (7)}
\end{enumerate}
then 
for any pair
$a,b \in \{r-1,r\}$ when $k = 0$,
and all distinct colours $a,b \in \{r-3,r-2,r-1,r\}$ when $k\geq 1$,
and for all paths $P$ in $\mathcal{G}^{(r)}_{k+1}$, the subgraphs induced by $V(P) \cap (V(G^{i,j}_{k,+}) \cup \{x_i,x_j\})$ (or $V(P)\cap (V(G^{i,j}_{k,-}) \cup \{y_i,y_j\})$),
such that $x_i$ is adjacent to the $A^{(r)}_{k+1}$ vertices in $V(G^{i,j}_{k,+})$,
admits
an $r$-colouring $\phi$
satsifying
\begin{enumerate}[(I)]
    \item if $k = 0$,  then for all $1\leq \ell \leq \lfloor \frac{r}{2} \rfloor - 1$, $\phi^{-1}(\ell) \subseteq A^{(r)}_1$, and \label{Lemma: Induction-colour-component (I)}
    \item if $k = 0$, then for all $\lfloor \frac{r}{2} \rfloor - 2\leq \ell \leq r - 2$, $\phi^{-1}(\ell) \subseteq B^{(r)}_1$, and \label{Lemma: Induction-colour-component (II)}
    \item if $k \geq 1$, then for all $1\leq \ell \leq \lfloor \frac{r}{2} \rfloor - k - 2$, $\phi^{-1}(\ell) \subseteq A^{(r)}_{k+1}$, and \label{Lemma: Induction-colour-component (III)}
    \item if $k \geq 1$, then for all $\lfloor \frac{r}{2} \rfloor - k - 1\leq \ell \leq r - 2k - 4$, $\phi^{-1}(\ell) \subseteq B^{(r)}_{k+1}$, and \label{Lemma: Induction-colour-component (IV)}
    \item if $x_i$ ($y_i$) is on the path $P$, then $\phi(x_i) = a$, and \label{Lemma: Induction-colour-component (V)}
    \item if $x_j$ ($y_j$) is on the path $P$, then $\phi(x_j) = b$, and \label{Lemma: Induction-colour-component (VI)}
    \item if $k \leq \lfloor \frac{r}{2} \rfloor -3$, then
        the vertices of 
        $V(P) \cap \left( W^{(r)}_{k+1} \cup \{v: v \text{ is minimal in } \preceq\}\right)$ are a colour class in $\phi$,
        otherwise $k = \lfloor \frac{r}{2} \rfloor -2$ and the vertices of 
        \[
        \{x_i\} \cup \Big(V(P) \cap \left( W^{(r)}_{k+1} \cup \{v: v \text{ is minimal in } \preceq\}\right)\Big)
        \]
        are a colour class in $\phi$. \label{Lemma: Induction-colour-component (VII)}
\end{enumerate}
\end{lemma} 

\begin{proof}
Suppose that for all paths $P'$ in $\mathcal{G}^{(r)}_k$ there is an $r$-colouring satisfying (\ref{Lemma: Induction-colour-component (1)})-(\ref{Lemma: Induction-colour-component (7)}).
Let $P$ be a fixed but arbitrary path in $\mathcal{G}^{(r)}_{k+1}$. 
Suppose without loss of generality that $P$ is 
maximal in the sense that no vertex
in $V(\mathcal{G}^{(r)}_{k+1})\setminus W^{(r)}_{k+1}$
can be added to $P$ making it a longer path.
We allow for the addition of vertices from $W^{(r)}_{k+1}$ to make $P$ a longer path.
Also suppose for the sake of contradiction that there is no colouring $\phi$ of the vertices of $P$ that satisfies (\ref{Lemma: Induction-colour-component (I)})-(\ref{Lemma: Induction-colour-component (VII)}). 
Let $1 \leq i,j \leq r+k-1$ be a fixed but arbitrary
pair of distinct integers,
without loss of generality we consider $G^{i,j}_{k,+}$.
Furthermore, suppose without loss of generality $N(x_i)\cap V(G)=A_{k,+}^{i,j}$ and $N(x_j)\cap V(G)=B_{k,+}^{i,j}$.

Let 
$$S = V(P) \cap \Big(V(G^{i,j}_{k,+})  \cup \{x_i,x_j\}\Big),$$
let $G = G^{(r)}_{k+1}[V(G^{i,j}_{k,+}) \cup \{x_i,x_j\}]$, and let $G' = G^{(r)}_{k+1}[S]$.
It is trivial to verify that $G$ is isomorphic to $\mathcal{G}^{(r)}_{k} -(v^A,v^B)$.
Without loss of generality suppose that $N(x_i) = A^{i,j}_{k,+}$ and $N(x_j) = B^{i,j}_{k,+}$ in $G$. 
The most obvious isomorphism between these graphs sends every $v$ in $G^{(r)}_{k}$
to $v$ in $G$, and sends $v^A$ to $x_i$ and $v^B$ to $x_j$.
Let $\tau$ denote this isomorphism from $\mathcal{G}^{(r)}_{k} - (v^A,v^B)$ to $G$.

Since $G$ is isomorphic to a subgraph of $\mathcal{G}^{(r)}_{k}$,
if there exists a path $P'$ in $\mathcal{G}^{(r)}_{k}$ such that $S \subseteq \tau(V(P'))$, 
then the claim follows immediately from our assumption there exists a colouring $\phi'$ satisfying (\ref{Lemma: Induction-colour-component (1)})-(\ref{Lemma: Induction-colour-component (7)}).
Suppose then for the sake of contradiction that there is no path $P'$ 
in $\mathcal{G}^{(r)}_{k}$ such that $S \subseteq \tau(V(P'))$.

Since $P$ is a path in $\mathcal{G}^{(r)}_{k+1}$ 
and $S$ is the intersection of $P$ and a subgraph of $\mathcal{G}^{(r)}_{k+1}$, 
$S$ is spanned by a set of vertex disjoint subpaths $P_1,\dots, P_\ell$ of $P$.
By our assumption that there is no path $P'$ 
in $\mathcal{G}^{(r)}_{k}$ such that $S \subseteq \tau(V(P'))$
we note $\ell\geq 2$.
Without loss of generality suppose $P = v_1,\dots, v_z$
and suppose that $P_t$ contains vertices with smaller index than $P_{t+1}$ for every $t$.
Here each path $P_t = v_{p}, v_{p + 1}\dots, v_{v_{p + p'}}$ corresponds to a maximal segment of $P$ that is entirely contained in $G$,
such that the $v_{p-1}$ and $v_{p+p'+1}$ are not vertices in $G$.

Recall that $\{v^A,v^B,x_i,x_j\}$ is a vertex cut in  $\mathcal{G}^{(r)}_{k+1}$ 
that separates  $G^{i,j}_{k,+}$ from the rest of $\mathcal{G}^{(r)}_{k+1}$.
It follows that if $P_t$ does not contain an endpoint $v_1$ or $v_z$ of $P$, then 
$v_p \in \{x_i,x_j\}$ or $v_{p-1} \in \{v^A,v^B\}$,
with the same being true for $v_{p+p'}$ and $v_{p+p'+1}$ respectively.
Since $\ell\geq 2$, there is no path $P_t$ which contains both endpoints of $P$.
Hence, if $P_t$ contains exactly one endpoint of $P$, say $v_1 = v_p$ without loss of generality,
then $v_{p+p'} \in \{x_i,x_j\}$ or $v_{p+p'+1} \in \{v^A,v^B\}$ must be the case.
Thus, all the paths $P_1,\dots, P_\ell$ 
being vertex disjoint and maximal implies that $\ell \leq 3$
with equality only if $v_1 \in V(P_1)$ and $v_{z} \in V(P_3)$

\vspace{0.25cm}
\noindent\underline{Case.1:} $\ell = 2$.
\vspace{0.25cm}

Let $P_1 = v_{p},\dots, v_{p+p'}$ and let $P_2 = v_{q},\dots, v_{q+q'}$.
We first consider the case where $P_1$ has an endpoint equal to $x_i$ or $v^A \in \{v_{p-1}, v_{p+p'+1}\}$,
and $P_2$ has an endpoint equal to $x_j$ or $v^B \in \{v_{q-1}, v_{q+q'+1}\}$.
Without loss of generality suppose $v_{p} = x_i$ and  $v_q = x_j$.
All other cases of this type follow via the same argument.
Then 
\[
P' = v_{p+p'}, \dots, v_{p+1}, v^A, v^B, v_{q+1}, \dots, v_{q+q'}
\]
is a path in $\mathcal{G}^{(r)}_{k}$.
But this is a contradiction since $S \subseteq \tau(V(P'))$.

Otherwise, the fact that $P_1$ and $P_2$ are vertex disjoint maximal segments of $P$ contained in $G$ implies,
without loss of generality, 
$P_1$ has an endpoint equal to $x_i$ or $v^A \in \{v_{p-1}, v_{p+p'+1}\}$, and 
$P_2$ has an endpoint equal to $x_i$ or $v^A \in \{v_{q-1}, v_{q+q'+1}\}$.
Since $P_1$ and $P_2$ are vertex disjoint, suppose $v_p = x_i$ and $v_{q-1} = v^A$ without loss of generality.
Then, 
\[
P' = v_{p+p'}, v_{p+p'-1}, \dots,v_{p+1}, v^A, v_{q}, \dots, v_{q+q'}
\]
is a path in $\mathcal{G}^{(r)}_{k}$.
But this is a contradiction as $S \subseteq \tau(V(P'))$.
Suppose then that $\ell = 3$.
\hfill $\diamond$

\vspace{0.25cm}
\noindent\underline{Case.2:} $\ell = 3$.
\vspace{0.25cm}

Let $P_1 = v_{1},\dots, v_{s}$, $P_2 = v_{p},\dots, v_{p+p'}$, $P_3 = v_{q},\dots, v_{z}$, and $s < p < q$.
Then, the reader can easily verify that
\[
\{v^A,v^B,x_i,x_j\} \subseteq \{v_s,v_{s+1},v_{p-1},v_p,v_{p+p'},v_{p+p'+1}, v_{q-1}, v_q,\}.
\]
We will consider different cases based on how this containment is structured.

We first consider the case where 
$\{v_{p-1},v_p,v_{p+p'},v_{p+p'+1}\}$ has non-empty intersection with $\{v^A,x_i\}$ and $\{v^B,x_j\}$.
Suppose without loss of generality that $v_p = x_i$ and $v_{p+p'} = x_j$. 
All other situations in this case will follow by similar arguments.
Then, $v_{s+1} = v^A$ and $v_{q-1} = v^B$, or vice versa.
Say $v_{s+1} = v^A$ and $v_{q-1} = v^B$ without loss of generality.
Then, 
\[
P' = v_1, \dots, v_{s}, v^A, v_{p+1}, \dots, v_{p+p'-1}, v^B, v_{q+1}, \dots, v_z
\]
is a path in  $\mathcal{G}^{(r)}_{k}$.
But this is a contradiction as $S \subseteq \tau(V(P'))$.
Suppose then that $\{v_{p-1},v_p,v_{p+p'},v_{p+p'+1}\}$ has empty intersection with $\{v^A,x_i\}$ or $\{v^B,x_j\}$. 
Note that this set cannot have empty intersection with both $\{v^A,x_i\}$ or $\{v^B,x_j\}$ 
since $P$ must cross the cut $\{v^A,v^B,x_i,x_j\}$ to enter $G$, before traveling along $P_2$,
and $P$ must cross the cut again to leave $G$ at the end of $P_2$. 
For the same reason, in what follows we must have $\{v^A,x_i\}\subseteq\{v_{p-1},v_p,v_{p+p'},v_{p+p'+1}\}$ or $\{v^B,x_j\}\subseteq\{v_{p-1},v_p,v_{p+p'},v_{p+p'+1}\}$.

\vspace{0.25cm}
\noindent\underline{SubCase.2.1:} $\{v_{p-1},v_p,v_{p+p'},v_{p+p'+1}\}$ has empty intersection with $\{v^A,x_i\}$.
\vspace{0.25cm}

Then, $\{v^B,x_j\} \subseteq \{v_{p-1},v_p,v_{p+p'},v_{p+p'+1}\}$.
We assume that $v_p = x_j$, that $v_{p+p'+1} = v^B$, as well as $v_s = x_i$, and $v_{q-1} = v^A$.
All other cases are equivalent.

If $v_1 \in B^{(r)}_{k+1}$, then 
\[
P' = v_z, v_{z-1},\dots, v_{q}, v^A, v_{s-1}, \dots, v_{1}, v^B, v_{p+1}, \dots, v_{p+p'}
\] 
is a path in $\mathcal{G}^{(r)}_{k}$ such that $S \subseteq \tau(V(P'))$.
By a similar argument, if $v_z \in B^{(r)}_k$, then there exists a path $P'$ in $\mathcal{G}^{(r)}_{k}$ such that $S \subseteq \tau(V(P'))$.
Since this is a contradiction, suppose $v_1, v_z \in A^{(r)}_{k+1}$.

Suppose for the sake of a contradiction that $N(v_1) \cap V(P_2) \neq \emptyset$.
Then there exists a vertex $v_{p+w} \in V(P_2)$ and an edge $(v_1,v_{p+w})\in E(G^{(r)}_k)$.
Then, 
$$
P' = v_z, v_{z-1},\dots, v_{q+1}, v^A, v_{s-1}, \dots, v_{1}, v_{p+w}, v_{p+w+1}, \dots, v_{p+p'}, v^B, v_{p+1}, \dots, v_{p+w-1}
$$ 
is a path in $\mathcal{G}^{(r)}_{k}$ such that $S \subseteq \tau(V(P'))$.
Since this is a contradiction, suppose $v_1$ has no neighbour in $P_2$.

Let $t$ be the least integer such that $v_t \in B^{(r)}_{k+1}$.
Since $v_{s}$ is equal to $x_i \in B^{(r)}_{k+1}$ and $v_1 \in A^{(r)}_{k+1}$, we know that $2 \le t \le  s$. 
Let $C = \{v_1,\dots, v_{t-1}\}$.
Then, $C \subseteq A^{(r)}_{k+1}$.
Hence, Lemma~\ref{Lemma: Induced As} implies $C$ induces a clique.

If $t = s$, then $V(P_1)\setminus \{v_s\} = C$ and
$$
P' = v_z, v_{z-1},\dots, v_{q+1}, v^A, v^B, v_{p+1}, \dots, v_{p+p'}
$$ 
is a path is a path in $\mathcal{G}^{(r)}_{k}$ such that $S\setminus C \subseteq \tau(V(P'))$.
Similarly, if $t < s$, then 
$$
P' = v_z, v_{z-1},\dots, v_{q}, v^A, v_{s-1}, \dots, v_{t}, v^{B}, v_{p+1}, \dots, v_{p+p'}
$$ 
is a path in $\mathcal{G}^{(r)}_{k}$ such that $S\setminus C \subseteq \tau(V(P'))$, as $v_t \in B^{(r)}_{k+1}$.
Thus, there exists a path $P'$ in $\mathcal{G}^{(r)}_{k}$ such that $S\setminus C \subseteq \tau(V(P'))$.

Suppose for the sake of a contradiction that $N(C) \cap A^{(r)}_{k+1} \cap S \neq \emptyset$.
Then there exists a vertex $v_g \in A^{(r)}_{k+1} \cap S$ where $g>t$, such that $(v_g,v_h)$ is an edge for some $1\leq h < t$. 
Then Lemma~\ref{Lemma: Induced As} implies $C\cup \{v_g\}$ induces a clique.

Recalling that $v_z,v_q$ are both in $A^{(r)}_{k+1}$,
if $g = z$ or $g = q$, then trivially, we can extend the path $P'$ (possibly by reversing the segment $v_z, v_{z-1},\dots, v_{q+1}$) to include all vertices in $C$, which is a contradiction.
Further, recall that $v_p,v_{p+p'}\in B^{(r)}_{k+1}$, so $g\neq p$ and $g\neq p+p'$.
Thus, $u$ has neighbours $v_{g-1}, v_{g+1}$ on $P'$.

Since $v_{g+1} \in V(G^{i,j}_{k,+})\cup \{x_i,v^A\}$, Lemma~\ref{Lemma: Bs next to As} implies $v_{g+1}$ must be adjacent to every vertex in $C\cup \{v_g\}$.
Then, 
$P'$ can be extended to a path $P''$
satisfying $S \subseteq \tau(V(P''))$
by traveling along $P'$ until reaching $v_g$, then taking a detour to collect all vertices in $C$, which can be done as $C\cup \{v_g\}$ is a clique,
before moving to $v_{g+1}$ which is possible since $C\cup \{v_g\} \subseteq N(v_{g+1})$, and continuing along $P'$.
This is a contradiction, so we suppose $N(C) \cap A^{(r)}_{k+1} \cap S = \emptyset$.

Therefore, $N(C) \subseteq B^{(r)}_k \cup \{v^A,x_i\}$.
Since there is a path $P'$ in $\mathcal{G}^{(r)}_{k}$ satisfying $S\setminus C \subseteq \tau(V(P'))$,
there exists an $r$-colouring of $G^{(r)}_{k+1}[S\setminus C]$ satisfying  (\ref{Lemma: Induction-colour-component (I)})-(\ref{Lemma: Induction-colour-component (VII)}) from the induction hypothesis.
Since $a'$ and $b'$ can be any distinct colours, subject to the lemma statement, we may
let $\phi$ be such an $r$-colouring,
where conditions (\ref{Lemma: Induction-colour-component (V)}) and (\ref{Lemma: Induction-colour-component (VI)}) imply $\phi(x_i) = a$ and $\phi(x_j) = b$, for distinct colours $a$ and $b$ of our  choosing (subject to the lemma statement).

If $|C|=1$, then $\{v_1\} = C$.
Recall $x_i,x_j \in S\setminus C$.
Note that $x_j$ is a non-neighbour of $v_1$ and every vertex in $u \in B^{(r)}_{k+1} \cap S$ distinct from $x_j$ is a neighbour of $x_j$.
Hence, for all vertices  $u \in B^{(r)}_{k+1} \cap S$ distinct from $x_j$, $\phi(u) \neq b$.
Since no neighbour of $C$ is in $A^{(r)}_{k+1}$, 
$N(v_1) \cap S \subseteq B^{(r)}_{k+1} \cap S$.
Thus, we can extend $\phi$ to colour $v_1$ by setting $\phi(v_1) = b$.
It is easy to verify that the extended colouring $\phi$ satisfies (\ref{Lemma: Induction-colour-component (I)})-(\ref{Lemma: Induction-colour-component (VII)}).

Consider then the case where $|C|>1$.
Then Lemma~\ref{Lemma: Induced As} implies $|C| \leq \lfloor \frac{r}{2} \rfloor$.
Recalling conditions (\ref{Lemma: Induction-colour-component (I)})) and (\ref{Lemma: Induction-colour-component (III)})
\begin{itemize}
    \item If $k = 0$, then condition (\ref{Lemma: Induction-colour-component (I)}) for the $r$-coloring $\phi$ of $G^{(r)}_{k+1}[S\setminus C]$ implies there are $\lfloor \frac{r}{2} \rfloor -1$ colours available to the vertices of $C$, since $C \subseteq A^{(r)}_{k+1}$ is not adjacent to any already coloured vertices from $A^{(r)}_{k+1}$. In particular, colours $\{1,\dots, \lfloor \frac{r}{2} \rfloor -1\}$ are available to all vertices of $C$.
    As in the case $|C| = 1$, (\ref{Lemma: Induction-colour-component (V)}) and (\ref{Lemma: Induction-colour-component (VI)}) imply colour $b > \lfloor \frac{r}{2} \rfloor -1$ is also available to every vertex of $C$.
    Hence, every vertex in $C$ can receive $\lfloor \frac{r}{2} \rfloor$ colours in an extension of $\phi$.
    Since $|C| \leq  \lfloor \frac{r}{2} \rfloor$ this implies we can extend $\phi$ to include the vertices of $C$.
    It is easy to verify that such an extended colouring $\phi$ satisfies (\ref{Lemma: Induction-colour-component (I)})-(\ref{Lemma: Induction-colour-component (VII)}).

    \item If $k\geq 1$, then condition (\ref{Lemma: Induction-colour-component (III)}) for the $r$-coloring $\phi$ of $G^{(r)}_{k+1}[S\setminus C]$ implies there are $\lfloor \frac{r}{2} \rfloor - k - 2$  colours available to the vertices of $C$, since $C$ is not adjacent to any already coloured vertices from $A^{(r)}_{k+1}$. 
    These are colours $\{1,\dots, \lfloor \frac{r}{2} \rfloor - k - 2\}$.
    By the same argument as when $k = 0$, colour $b > \lfloor \frac{r}{2} \rfloor - k - 2$ is also free.
    So there are at least $\lfloor \frac{r}{2} \rfloor - k - 1$ colours free to use on vertices of $C$.

    We will now consider which of the $2k+4$ colours from the set $\{r - 2k - 3,r - 2k - 2, \dots, r\}$
    can be used to colour vertices of $C$. 
    One of these colours is $b$ which we have already counted.
    It also may be the case that $a\in\{r - 2k - 3,r - 2k - 2, \dots, r\}$, and clearly $a$ cannot be used to colour a vertex in $C$ by condition (\ref{Lemma: Induction-colour-component (V)}).
    So there are $2k+2$ colours in  $\{r - 2k - 3,r - 2k - 2, \dots, r\}$, which may or may not be available to the vertices in $C$.
    Since there are already $\lfloor \frac{r}{2} \rfloor - k - 1$ colours free to use on vertices of $C$, where $|C| \leq \lfloor \frac{r}{2}\rfloor$,
    if an additional $k+1$ colours are free, then we can extend $\phi$ to $C$ using colours  from $\{1, \dots, \lfloor \frac{r}{2} \rfloor - k - 2\}$ and from $\{r - 2k - 3,r - 2k - 2, \dots, r\}$.
    
    If $\phi$ can be extended to the vertices of $C$ so that $C$ receives only colours 
    from $\{1, \dots, \lfloor \frac{r}{2} \rfloor - k - 2\}$ and from $\{r - 2k - 3,r - 2k - 2, \dots, r\}$,
    then it is easy to verify that conditions (\ref{Lemma: Induction-colour-component (I)})-(\ref{Lemma: Induction-colour-component (VII)}) are satisfied.
    This would contradict our assumption that no such colouring exists for $P$.
    Suppose then that $\phi$ cannot be extended this way.

    Then, there are at least $k+2$ distinct vertices $b_1,\dots, b_{k+2} \in B^{(r)}_{k+1} \cap N(C) \cap S$
    which all receive distinct colours from  $\{r - 2k - 3,r - 2k - 2, \dots, r\}$, none of which are colour $a$ or $b$.
    Let $\{b_1,\dots, b_{k+2} \}$ be one such set of vertices.
    Then, without loss of generality with respect to the labeling of the $b_d$-vertices, 
    Lemma~\ref{Lemma: k+2 Bs next to a common A vertex} implies vertices $b_1$ and $b_2$ satisfy
    $N[b_1] = N[b_2]$ and, given that $C$ has no neighbours in $A^{(r)}_{k+1}\cap S$, we must have $N(b_1)\cap A^{(r)}_{k+1} \cap S = C$.

    Since $N(v_1) \cap V(P_2) = \emptyset$ and $C$ is a clique,
    Lemma~\ref{Lemma: Bs next to As} implies that $b_1,b_2 \notin V(P_2)$.
    Thus, $b_1,b_2 \in V(P_1)\cup V(P_3)$.
    Observe that $b_1 = v_g$ and $b_2 = v_h$ for some $g$ and $h$.
    Without loss of generality suppose $g<h$.

    Suppose $v_g,v_h \in V(P_1)$.
    Then, $t\leq g < h < s$.
    Since $N(b_1)\cap A^{(r)}_{k+1} \cap S = C$,
    every vertex in $N[b_1] (= N[b_2])$ that is not also in $C\cup \{x_i,x_j,v^A,v^B\}$ must be in $B^{(r)}_{k+1}$.
    Hence, $v_{h-1} \in B^{(r)}_{k+1}$.
    Then,
    $$
    P'' = v_z,\dots, v_{q}, v^A, v_{s-1}, \dots, v_h, v_1, v_2,\dots, v_{h-2}, v_{h-1}, v^B, v_{p+1}, \dots, v_{p+p'}
    $$
    is a path in is a path in $\mathcal{G}^{(r)}_{k}$ such that $S\subseteq \tau(V(P''))$.
    But the existence of such a path $P''$ is a contradiction.
    So we suppose $v_g$ and $v_h$ are not both in $P_1$.

    Suppose $v_g,v_h \in V(P_3)$.
    Then, $q<g < h < z$ since we previously showed $v_z \in A^{(r)}_{k+1}$ whereas $v_h \in B^{(r)}_{k+1}$.
    By the same argument as before, $v_{g+1}, v_{h-1} \in B^{(r)}_{k+1}$.
    Then, 
    $$
    P'' = v_s,\dots, v_1,v_h, v_{h+1}, \dots, v_{z}, v^A, v_{q}, \dots, v_{h-1}, v^B, v_{p+1},\dots, v_{p+p'}
    $$
    is a path in is a path in $\mathcal{G}^{(r)}_{k}$ such that $S\subseteq \tau(V(P''))$.
    But the existence of such a path $P''$ is a contradiction.
    So we suppose $v_g$ and $v_h$ are not both in $P_3$.

    Suppose $v_g\in V(P_1)$ and $v_h \in V(P_3)$.
    Given $g< h$ and $v_g,v_h \in V(P_1)\cup V(P_3)$ this is the only remaining case.
    Then, $t\leq g < s < q< h < z$ by the same arguments seen before.
    As before, $v_{h-1} \in B^{(r)}_{k+1}$.
    Then, 
    \[
    P'' = v_{p+1},\dots, v_{p+p'},v^B, v_{h-1}, \dots, v_q, v^A, v_z, \dots, v_h, v_1, \dots, v_s
    \]
    is a path in is a path in $\mathcal{G}^{(r)}_{k}$ such that $S\subseteq \tau(V(P''))$.
    But the existence of such a path $P''$ is a contradiction.

    Thus, the existence of a set of distinct vertices  $\{b_1,\dots, b_{k+2} \} \subseteq B^{(r)}_{k+1} \cap N(C)$ is a contradiction.
    It follows no such set exists.
    This implies that $|B^{(r)}_{k+1} \cap N(C)| \leq k+1$.
    Hence, there are least $(2k+2)-(k+1) = k+1$ colours in $\{r - 2k - 3,r - 2k - 2, \dots, r\}$ distinct from $a$ and $b$, which are free to use when colouring vertices in $C$.
    Therefore, taking $b$ into account, there are at least $1+(k+1)+\lfloor\frac{r}{2}\rfloor-k-2=\lfloor \frac{r}{2}\rfloor$ colours that can be used to colour $C$. Again, since $|C|\le \lfloor \frac{r}{2}\rfloor$ by Lemma~\ref{Lemma: Induced As}, $\phi$ can be extended to colour $C$ using colours from 
    $\{1, \dots, \lfloor \frac{r}{2} \rfloor - k - 2\}$ and $\{r - 2k - 3,r - 2k - 2, \dots, r\}$
    while satisfying (\ref{Lemma: Induction-colour-component (I)})-(\ref{Lemma: Induction-colour-component (VII)}).
\end{itemize}

This concludes the proof of SubCase.2.1.
\hfill $\diamond$

\vspace{0.25cm}
\noindent\underline{SubCase.2.2:} $\{v_{p-1},v_p,v_{p+p'},v_{p+p'+1}\}$ has empty intersection with $\{v^B,x_j\}$.
\vspace{0.25cm}

Then, $\{v^A,x_i\} \subseteq \{v_{p-1},v_p,v_{p+p'},v_{p+p'+1}\}$.
We assume that $v_p = x_i$, that $v_{p+p'+1} = v^A$, as well as $v_s = x_j$, and $v_{q-1} = v^B$.
All other cases are equivalent.

Observe that if there is a path $P^*$ in $\mathcal{G}^{(r)}_{k+1}$ that such that 
\[
V(P^*) \cap  \Big(V(G^{i,j}_{k,+})  \cup \{x_i,x_j\}\Big) = S
\]
where $P^*$ falls under any case other than SubCase.2.2, then we have already shown there exists an $r$-colouring as required.
Suppose then that all paths $P^*$ in $\mathcal{G}^{(r)}_{k+1}$ which intersect $G$ in vertex set $S$ are in SubCase.2.2.
Further, we suppose without loss of generality that $P$ is chosen among all paths $P^*$ 
to minimize the least integer $t$ such that 
\[
v_t \in A^{(r)}_{k+1} \cup W^{(r)}_{k+1} \cup \{v: v \text{ is minimal in } \preceq\}.
\]
Such an integer $t$ exists since $v_{p-1}\in A^{(r)}_{k+1}$ as $v_p = x_i$. 
Note also that this shows that $t \leq p$.
Here minimizing $t$ amounts to having a vertex in the described set as close to the beginning of the path $P^*$ as possible. 
We will now show that $t>1$ by showing a contradiction in each of the three cases when $t=1$.

If $v_1 \in A^{(r)}_{k+1}$, then 
\[
P' = v_z ,\dots, v_q, v^B, v_{s-1},\dots, v_{1},v^A, v_{p+1}, \dots, v_{p+p'}
\]
is a path in is a path in $\mathcal{G}^{(r)}_{k}$ such that $S\subseteq \tau(V(P'))$.
But the existence of such a path $P'$ is a contradiction.
Suppose then that $v_1 \notin A^{(r)}_{k+1}$.

If $v_1 \in W^{(r)}_{k+1}$, then consider $P-\{v_1\}$.
If $P-\{v_1\}$ is maximal in the sense that no vertex
in $V(\mathcal{G}^{(r)}_{k+1})\setminus W^{(r)}_{k+1}$
can be added to $P$ making it a longer path,
then begin the argument for SubCase.2.2 again with $P-\{v_1\}$ as our path.
If $P-\{v_1\}$ is not maximal in this sense, then extend $P-\{v_1\}$ into a maximal path $P'$
and begin the argument for SubCase.2.2 again with this as your choice of path.
Notice that if this process must be repeated to form a new path $P''$, then $S\setminus \{v_1\} \subseteq V(P'')$.
Furthermore, repeating this process again will never revisit $v_1$, since all neighbours of $v_1$ are in $S\setminus \{v_1\}$,
given the maximiality of $P$ and 
that Lemma~\ref{Lemma: W structure} implies all neighbours of $v_1$ are in $V(\mathcal{G}^{(r)}_{k+1})\setminus W^{(r)}_{k+1}$.
Trivially, the resulting paths will still be in SubCase.2.2.

Suppose now without loss of generality that $P'''$ is the path obtained from applying this process until 
$P'''$ is a maximal path, in the desired sense, under SubCase.2.2 and the first vertex on $P'''$ is not in $W^{(r)}_{k+1}$.
Then, the remainder of the proof of this SubCase shows that 
the graph induced by $V(P''')$ admits an  
$r$-colouring
satisfying (\ref{Lemma: Induction-colour-component (I)})-(\ref{Lemma: Induction-colour-component (VII)}).
Since $S\setminus \{v_1\} \subseteq V(P''')$ this implies 
the graph induced by $S\setminus \{v_1\}$ admits an  
$r$-colouring
satisfying (\ref{Lemma: Induction-colour-component (I)})-(\ref{Lemma: Induction-colour-component (VII)}).
Then condition (\ref{Lemma: Induction-colour-component (VII)}), Lemma~\ref{Lemma: W structure}, and Lemma~\ref{Lemma: B minus W structure}~(\ref{Lemma: B minus W structure (4)}),
imply that this colouring can be extended to $v_1$ while satisfying (\ref{Lemma: Induction-colour-component (I)})-(\ref{Lemma: Induction-colour-component (VII)}). 
This contradicts out assumption no such colouring exists so suppose $v_1\notin W^{(r)}_{k+1}$.

Finally, consider the case where $v_1$ is minimal in $\preceq$.
Then, from Lemma~\ref{Lemma: B minus W structure}~(\ref{Lemma: B minus W structure (2)}), there exists vertices $u_2,u_3,\ldots u_{\lfloor\frac{r}{2}\rfloor}$ such that
$v_1 = u_1 \prec \dots \prec u_{\lceil \frac{r}{2} \rceil}$ in a covered chain in $\preceq$ such that $u_1$ is $\preceq$ minimal.
From Lemma~\ref{Lemma: B minus W structure}~(\ref{Lemma: B minus W structure (1)}),
we have that $N(u_1) \cap B^{(r)}_{k+1}$ induces a clique, call it $K$.
Similarly, by Lemma~\ref{Lemma: B minus W structure}~(\ref{Lemma: B minus W structure (5)}), $N(u_1) \cap A^{(r)}_{k+1}$ induces a clique, call it $C$.
Then the maximality of $P$ implies $C\cup K \subseteq S$.

If there exists vertices $v_f, v_{f+1} \in C$, which are adjacent in $P$,
then 
\[
P' = v_f, v_{f-1}, v_{f-2,}\dots, v_2,v_1,v_{f+1}, v_{f+2}, \dots, v_z
\]
is a path spanning $S$ with an endpoint in $A^{(r)}_{k+1}$. If $P'$ is not in SubCase.2.2, then we are done by Case.1 or Case.2.1, so suppose $P'$ is in SubCase.2.2.
Thus, by the same argument as we already made for when $v_1\in A^{(r)}_{k+1}$, this implies a contradiction.
So suppose there is no index $f$ where $v_f, v_{f+1} \in C$.

By Lemma~\ref{Lemma: B minus W structure}~(\ref{Lemma: B minus W structure (5)}), $|C| = \lfloor \frac{r}{2} \rfloor$.
Notice that if $v_z \in C$, then by the same argument as $v_1\in A^{(r)}_{k+1}$, this leafs to a contradiction.
So $v_z \notin C$.
Then, no edge induced by $C$ is on the path $P$, and the endpoints of $P$ are not on $C$.
Hence, 
there must be $|C| + 1 = \lfloor \frac{r}{2}\rfloor  + 1$ edges in $E(P) \cap E(C, S\setminus C)$.
Thus, there are vertices $b_1,\dots, b_{\
\lfloor \frac{r}{2}\rfloor  + 1} \in S\cap B^{(r)}_{k+1}$
all of which are adjacent to vertices in $C$ using edges from $E(P)$.
Given, $k \leq \lfloor \frac{r}{2} \rfloor -1$ this implies there are at least $k+2$ vertices $b_1,\dots, b_{k+2}$.
Then by the same argument as the proof of
Lemma~\ref{Lemma: k+2 Bs next to a common A vertex}, we can conclude that at least two of these vertices, 
say $b_1$ and $b_2$ without loss of generality,
are in $\{u_1,\dots, u_{\lceil \frac{r}{2} \rceil} \}$.

Therefore, there exists a vertex $v_f = u_d \neq u_1 = v_1$, such that $v_{f-1} \in C$ or $v_{f+1} \in C$. 
By Lemma~\ref{Lemma: B minus W structure}~(\ref{Lemma: B minus W structure (3)}), $N[v_f] = N[v_1]$.

Also, there exists a second vertex $v_g = u_{d'}$, which may equal $v_1$, such that $v_{g-1} \in C$ or $v_{g+1} \in C$.
If $v_g \neq v_1$, then Lemma~\ref{Lemma: B minus W structure}~(\ref{Lemma: B minus W structure (3)}) implies $N[v_g] = N[v_1]$.
Since we allow $v_g$ to equal $v_1$, we assume $f > g$ when $v_g \neq v_1$.

If $v_{f-1}\in C$ or $v_{g-1}\in C$ (we note that if $g=1$, $v_{g-1}$ does not exist), then 
\[
P' = v_z, v_{z-1}, \dots,v_f,v_1,v_{2}, \dots, v_{f-1}
\]
or
\[
P' = v_z, v_{z-1}, \dots, v_g, v_1, v_2,\dots, v_{g-1}
\]
is a path spanning $S$ with an endpoint in $A^{(r)}_{k+1}$.
If $P'$ is not in SubCase.2.2, then we are done by Case.1 or Case.2.1, so suppose $P'$ is in SubCase.2.2.
Thus, by the same argument as we already made for when $v_1\in A^{(r)}_{k+1}$, this implies a contradiction.
Suppose then that $v_{f-1} \notin C$ and $v_{g-1}\not\in C$. This implies $v_{f+1} \in C$ and $v_{g+1} \in C$.

Then,
\[
P' = v_z, v_{z-1}, \dots, v_{f+1}, v_1,\dots,v_g, v_f, v_{f-1}, \dots, v_{g+1}
\]
is a path, in SubCase.2.2, spanning $S$ with an endpoint in $A^{(r)}_{k+1}$.
If $P'$ is not in SubCase.2.2, then we are done by Case.1 or Case.2.1, so suppose $P'$ is in SubCase.2.2.
Thus, by the same argument as we already made for when $v_1\in A^{(r)}_{k+1}$, this implies a contradiction.
Therefore, we have contradicted $v_1$ is minimal in $\preceq$.
Suppose then that $v_1$ is not minimal in $\preceq$.

Hence, we have shown that $v_1$ is not in $A^{(r)}_{k+1}$, or in $W^{(r)}_{k+1}$, and $v_1$ is not minimal in $\preceq$.
Suppose then that $t>1$.

Let $u_1 \prec \dots \prec u_{\lceil \frac{r}{2} \rceil}$ be a covered chain in $\preceq$ such that $u_1$ is $\preceq$ minimal, and $u_1 \prec v_1$.
Recall that in the definition of $\preceq$ in the proof of Lemma~\ref{Lemma: B minus W structure} starts with a total ordering of the $\lceil \frac{r}{2} \rceil$ vertices in $G^{(r)}_0$.
Therefore, such a covered chain exists.
Further, recall that each time $x_i$ and $x_j$ and (similarly $y_i$ and $y_j$) are added in constructing $G^{(r)}_{\ell+1}$ from $G^{(r)}_{\ell}$ we have $u\preceq x_i$ and $u\preceq x_j$ for all $u\in G^{i,j}_k$.
Thus, we may choose $u_1$ such that $u_1\preceq v_1$.
Since, $v_1$ is not minimal in $\preceq$ note that $v_1 \neq u_1$.
Again from the definition of $\preceq$,
we have that $N(u_1) \cap B^{(r)}_{k+1}$ induces a clique, call it $K$, and $v_1 \in K$.
Similarly, by Lemma~\ref{Lemma: B minus W structure}~(\ref{Lemma: B minus W structure (5)}), $N(u_1) \cap A^{(r)}_{k+1}$ induces a clique call it $C$.

Since, $P$ is maximal, $u_1 \in S$, otherwise, $P$ can be made longer by adding $u_1$. 
Let $h\in \{1,\ldots,z\}$ be the index of $u_1$ on $P$, i.e. $u_1=v_h$.
Since $u_1 \in B^{(r)}_{k+1}$ and $v_p = x_i$ and $v_{p+p'+1} = v^A$, $h \neq p$, $h\neq p+1$, and $h\neq p+p'$.
We have already shown that $u_1\neq v_1$, so $h\neq 1$.
If $h = z$, then by reversing the index of $P$, there exists a path $P^*$ whose first vertex is $\preceq$ minimal.
This contradicts the minimality of $t>1$.
Thus, there are vertices $v_{h-1},v_{h+1} \in S \cup \{x_j,v^B\}$.

If $v_{h-1}, v_{h+1}$ are adjacent,
then
\[
P^* = v_h, v_1,\dots, v_{h-1}, v_{h+1}, \dots, v_z.
\]
is a path in $\mathcal{G}^{(r)}_{k+1}$ which intersects  $(V(G^{i,j}_{k,+})  \cup \{x_i,x_j\})$ in $S$.
But this is a contradiction, since $v_h = u_1$ where $u_1$ is minimal in $\preceq$,
contradicting the minimality of $t>1$.
Thus, $v_{h-1}, v_{h+1}$ are non-adjacent. 
Since $K\cup \{v^B,x_j\}$ and $C$ both induce cliques,
$v_{h-1}$ and $v_{h+1}$ being non-adjacent
forces 
$v_{h-1}$ or $v_{h+1}$ is in $C$, 
while the other is in  $K\cup \{v^B,x_j\}$.

If $v_{h-1} \in C$ and $v_{h+1} \in K \cup \{v^B,x_j\}$,
then
\[
P^* = v_2,\dots, v_{h-1},v_h, v_1, v_{h+1}, \dots, v_z.
\]
is a path in $\mathcal{G}^{(r)}_{k+1}$ which intersects  $(V(G^{i,j}_{k,+})  \cup \{x_i,x_j\})$ in $S$
because $v_1,v_h,v_{h+1} \in K\cup \{v^B,x_j\}$, a clique.
But this contradicts the minimality of $t$, since $t\le h$ as $v_{h} = u_1$ is $\preceq$ minimal.
Suppose then that $v_{h+1} \in C$ and $v_{h-1} \in K \cup \{v^B,x_j\}$.

If $t \leq h-1$, 
then
\[
P^* = v_2,\dots, v_{h-1},v_1,v_h,v_{h+1}, \dots, v_z.
\]
is a path in $\mathcal{G}^{(r)}_{k+1}$ which intersects  $(V(G^{i,j}_{k,+})  \cup \{x_i,x_j\})$ in $S$.
But this contradicts the minimality of $t$.

Thus, we may suppose that $t \ge h$, and therefore $t=h$ since $v_h=u_1$ is minimal in $\preceq$.
Suppose that $u_x = v_g$ for some $1 < x \leq \lceil \frac{r}{2} \rceil$ and $1\leq g < h$.
Then, Lemma~\ref{Lemma: B minus W structure}~(\ref{Lemma: B minus W structure (3)}) implies that $N[v_g] = N[v_h]$.
Recalling $v_1$ and $v_h$ are adjacent, we have that
\[
P^* = v_{g+1}, \dots, v_h,v_{1},\dots, v_{g},v_{h+1}, \dots, v_z.
\]
is a path in $\mathcal{G}^{(r)}_{k+1}$ which intersects  $(V(G^{i,j}_{k,+})  \cup \{x_i,x_j\})$ in $S$.
But this contradicts the minimality of $t$, since $|\{g+1,\dots, h\}| < t=h$.

Suppose then that for all $1\leq g < h$, $v_g \neq u_x$ for any $1 < x \leq \lceil \frac{r}{2} \rceil$.
By the maximality of $P$, $K\cup \{x_j,v^B\} \subseteq S$.
Hence, every vertex $u_x \in S$.
Let $u_x = v_{f}$ be a fixed but arbitrary vertex in the covered chain.
Then $f>h = t$.

If $v_{f-1}, v_{f+1}$ are adjacent, then
\[
P^* = v_h,v_1,v_{2}, \dots, v_{h-1}, v_f, v_{h+1}, \dots, v_{f-1},v_{f+1},\dots, v_z.
\]
is a path in $\mathcal{G}^{(r)}_{k+1}$ which intersects  $(V(G^{i,j}_{k,+})  \cup \{x_i,x_j\})$ in $S$.
But this contradicts the minimality of $t$, since $t>1$.

Otherwise, by the same argument as $v_{h-1}$ and $v_{h+1}$, one of  $v_{f-1}, v_{f+1}$ is in $C$ and the other is in $K\cup \{x_j,v^B\}$.
If $v_{f-1} \in K\cup \{x_j,v^B\}$, let
\[
P^* = v_{2}, \dots, v_{f-1},v_1, v_f, \dots, v_z,
\] 
if $v_{f+1} \in K\cup \{x_j,v^B\}$ instead, then let 
\[
P^* = v_{2}, \dots, v_f, v_1, v_{f+1} \dots, v_z.
\] 
In either case $P^*$ is a path in $\mathcal{G}^{(r)}_{k+1}$ which intersects  $(V(G^{i,j}_{k,+})  \cup \{x_i,x_j\})$ in $S$.
But $f > h =t$ implies that we have contradicted the minimality of $t$.

This concludes the proof of SubCase.2.2.
\hfill $\diamond$
\vspace{0.25cm}

Therefore, we have shown if $\ell =1$ or $2$ or $3$, then we reach a contradiction.
Recall that we have demonstrated $1 \leq \ell\leq 3$.
Hence, we have proven a contradiction in each possible case.
It follows that there exists a colouring $\phi$ satisfying (\ref{Lemma: Induction-colour-component (I)})-(\ref{Lemma: Induction-colour-component (VII)}).
This concludes the proof.
\end{proof}

We are now prepared to prove $r(G^{(r)}_k) = r$.
This will conclude the proof of Theorem~\ref{Thm: Lower result}.

\begin{theorem}\label{Thm: Technical version of lower result}
     Let $r\geq 6$ and $0 \leq k \leq \lfloor \frac{r}{2} \rfloor - 1$.
     Then, $r(G^{(r)}_k) = r$.
\end{theorem}

\begin{proof}
Instead of proving the theorem directly we will instead prove the following stronger statement.  
    Let $r\geq 6$ and $0 \leq k < \lfloor \frac{r}{2} \rfloor - 1$
    and
    let $\preceq$ be a partial order of $B^{(r)}_{\lfloor\frac{r}{2} \rfloor - 1} \setminus W^{(r)}_{\lfloor\frac{r}{2} \rfloor - 1}$
    satisfying the properties of Lemma~\ref{Lemma: B minus W structure}.
    Then for all paths $P$ in $\mathcal{G}^{(r)}_k$, and for all distinct colours $a,b \in \{r-1,r\}$ when $k = 0$,
    and all distinct colours $a,b \in \{r-3,r-2,r-1,r\}$ when $k\geq 1$,
    there is an $r$-colouring $\phi: V(P) \rightarrow [r]$ of $\mathcal{G}^{(r)}_k[V(P)]$
    such that 
    \begin{enumerate}
        \item if $k = 0$,  then for all $1\leq \ell \leq \lfloor \frac{r}{2} \rfloor - 1$, $\phi^{-1}(\ell) \subseteq A^{(r)}_0$, and \label{Thm: Technical version of lower result (1)}
        \item if $k = 0$, then for all $\lfloor \frac{r}{2} \rfloor - 2\leq \ell \leq r - 2$, $\phi^{-1}(\ell) \subseteq B^{(r)}_0$, and \label{Thm: Technical version of lower result (2)}
        \item if $k \geq 1$, then for all $1\leq \ell \leq \lfloor \frac{r}{2} \rfloor - k - 2$, $\phi^{-1}(\ell) \subseteq A^{(r)}_k$, and \label{Thm: Technical version of lower result (3)}
        \item if $k \geq 1$, then for all $\lfloor \frac{r}{2} \rfloor - k - 1\leq \ell \leq r - 2k - 4$, $\phi^{-1}(\ell) \subseteq B^{(r)}_k$, and \label{Thm: Technical version of lower result (4)}
        \item if $v^A$ is on the path $P$, then $\phi(v^A) = a$, and \label{Thm: Technical version of lower result (5)}
        \item if $v^B$ is on the path $P$, then $\phi(v^B) = b$, and \label{Thm: Technical version of lower result (6)}
        \item if $k \leq \lfloor \frac{r}{2} \rfloor -3$, then
        the vertices of 
        $V(P) \cap \left( W^{(r)}_{k} \cup \{v: v \text{ is minimal in } \preceq\}\right)$ are a colour class in $\phi$,
        otherwise $k = \lfloor \frac{r}{2} \rfloor -2$ and the vertices of 
        \[
        \{v^A\} \cup \Big(V(P) \cap \left( W^{(r)}_{k} \cup \{v: v \text{ is minimal in } \preceq\}\right)\Big)
        \]
        are a colour class in $\phi$. \label{Thm: Technical version of lower result (7)}
    \end{enumerate}
    Moreover, for all paths $P$ in $\mathcal{G}^{(r)}_{\lfloor \frac{r}{2} \rfloor - 1}$
    there exists an $r$-colouring of $G^{(r)}_{\lfloor \frac{r}{2} \rfloor - 1}[P\setminus \{v^A,v^B\}]$.

If $k = 0$, then $G^{(r)}_k$ is isomorphic to $K_r$ and the result is trivial.
Suppose then that $0 \leq k < \lfloor \frac{r}{2} \rfloor - 1$ and that 
for all paths $P'$ in $\mathcal{G}^{(r)}_k$
there is an $r$-colouring $\phi': V(P') \rightarrow [r]$ of $\mathcal{G}^{(r)}_k[V(P)]$
satisfying (\ref{Thm: Technical version of lower result (1)})-(\ref{Thm: Technical version of lower result (7)}).

Let $P$ be a fixed but arbitrary path in $\mathcal{G}^{(r)}_{k+1}$.
Given $P$ we define the graph $Q$ as follows.
Let $V(Q) = \{x_1,\dots, x_{r+k-1}\} \cup \{y_1,\dots, y_{r+k-1}\}$
and 
let $(x_i,x_j) \in E(Q)$ if and only if $P$ contains a vertex in $G^{i,j}_{k, +}$
and let $(y_t,y_\ell) \in E(Q)$ if and only if $P$ contains a vertex in $G^{i,j}_{k,-}$.
  {Recall that} we use plus and minus in the subscripts of $G^{i,j}_{k}$ to delineate if $G^{i,j}_{k, \pm}$ is a
subgraph of $H_{k+1}^{(r),+}$ or $H_{k+1}^{(r),-}$.
There are no edges of the form $(x_i,y_t)\in E(Q)$.

\vspace{0.25cm}
\noindent\underline{Claim:} $\chi(Q) \leq 4$.
\vspace{0.25cm}

Since there are no edges of the form $(x_i,y_t)\in E(Q)$, it is sufficient to prove both 
$Q[\{x_1,\dots, x_{r+k-1}\}]$ and $Q[\{y_1,\dots, y_{r+k-1}\}]$ are $4$-colourable.
Without loss of generality consider $Q' = Q[\{x_1,\dots, x_{r+k-1}\}]$.
Recall that $\{x_i,x_j\}$ is a vertex cut that separates $G^{i,j}_{k,+}$ from the rest of $G^{(r)}_{k+1}$
and recall that $G^{(r)}_{k+1} = \mathcal{G}^{(r)}_{k+1} - \{v^A,v^B\}$.

Since $\{x_i,x_j\}$ is a vertex cut that separates $G^{i,j}_{k,+}$ from the rest of $G^{(r)}_{k+1}$,
if $P$ intersects a subgraph $G^{i,j}_{k, +}$,
which $P$ did not enter and exit through vertices $x_i$ and $x_j$,
then this subgraph contains an endpoint of $P$, or $P$
entered or exited this subgraph using vertices $\{v^A, v^B\}$.
Let $\{e_1,\dots, e_\ell\}$ be the set of all edges in $Q$ such that $e_q = (x_i,x_j)$
where $G^{i,j}_{k,+}$ contains an endpoint of $P$, or $P$ enters or exists $G^{i,j}_{k,+}$ using vertices 
$\{v^A, v^B\}$.
Given $P$ is a path, vertices $v^A,v^B$ can each be visited only once, and $P$ has at most $2$ endpoints.
Thus, $\ell \leq 6$.

Since $P$ is a path, the reader can easily verify $Q' - \{e_1,\dots, e_\ell\}$ is a forest of paths 
Hence, if $\ell \leq 5$, then Lemma~\ref{Lemma: Almost forest 4-colour} implies $\chi(Q')\leq 4$ as required.
Similarly, if $Q' - \{e_1,\dots, e_\ell\}$ is a forest with at least two connected components, 
then Lemma~\ref{Lemma: Worst case Almost forest 4-colour} implies $\chi(Q')\leq 4$ as required.
Suppose for the sake of contradiction that $\ell = 6$ and $Q' - \{e_1,\dots, e_6\}$ is connected.

Since $\ell = 6$, there are $6$ distinct subgraphs of $Q'$ defined by the edges $\{e_1,\dots, e_6\}$, call them $G^{i_1,j_1}_{k, +}, \dots, G^{i_6,j_6}_{k, +}$,
where each subgraph contains an endpoint of $P$ or the subgraph is entered or exited using vertices 
$v^A$ or $v^B$.
Without loss of generality,
suppose these subgraphs are visited by $P$ in the order of their indices, and
suppose $G^{i_1,j_1}_{k, +}$ and $G^{i_6,j_6}_{k, +}$ contain endpoints of $P$,
while $G^{i_2,j_2}_{k, +}$ is exited by $v^A$, $G^{i_3,j_3}_{k, +}$ is entered by $v^A$,
$G^{i_4,j_4}_{k, +}$ is exited by $v^B$, and $G^{i_5,j_5}_{k, +}$ is entered by $v^B$.
Then, we can consider three vertex-disjoint subpaths $P_1,P_2,P_3$ of $P$,
where $P_1$ begins in $G^{i_1,j_1}_{k, +}$ and ends when $P$ exits  $G^{i_2,j_2}_{k, +}$,
while $P_2$ begins in  $G^{i_3,j_3}_{k, +}$ and ends when $P$ exits  $G^{i_4,j_4}_{k, +}$,
and $P_3$ begins in $G^{i_5,j_5}_{k, +}$ and ends in $G^{i_6,j_6}_{k, +}$.

By the definition of $Q'$, the paths $P_1,P_2,P_3$ define vertex-disjoint connected subgraphs, say $Q'_1,Q'_2,Q'_3$,  of $Q' - \{e_1,\dots, e_6\}$ which partition the edges of $Q' - \{e_1,\dots, e_6\}$.
Thus, $Q' - \{e_1,\dots, e_6\}$ is disconnected.
But this is a contradiction.
This completes the proof of the claim.
\hfill $\diamond$
\vspace{0.25cm}

Suppose that $\psi$ is a $4$-colouring of $Q$ that uses colours $\{r-3,r-2,r-1,r\}$.
Recall that \[V(Q) = \{x_1,\dots, x_{r+k-1}\}\cup \{y_{1}, \dots, y_{r+k-1}\}.\]
We will extend $\psi$ to an $r$-colouring of $G^{(r)}_{k+1}[P-\{v^A,v^B\}]$ satisfying conditions (\ref{Thm: Technical version of lower result (1)})-(\ref{Thm: Technical version of lower result (4)}) and condition (\ref{Thm: Technical version of lower result (7)}) for parameter $k$.
By this we mean that the extended colouring $\psi$ will reserve colours for the sets $A^{(r)}_{k+1}$ and $B^{(r)}_{k+1}$ 
per conditions  (\ref{Thm: Technical version of lower result (1)})-(\ref{Thm: Technical version of lower result (4)}),
as if we are colouring a traceable subgraph of $\mathcal{G}^{(r)}_{k}$.
Moreover, we will ensure that the vertices in 
\[
\Big(V(P) \cap \left( W^{(r)}_{k} \cup \{v: v \text{ is minimal in } \preceq\}\right)\Big)
\]
form a colour class, again as if in a traceable subgraph of $\mathcal{G}^{(r)}_{k}$, per condition (\ref{Thm: Technical version of lower result (7)}).

Let $i,j$ be fixed but arbitrary and consider $G^{i,j}_{k,+}$ without loss of generality with respect to $G^{i,j}_{k,\pm}$.
Then  $\psi(x_i)$ and $\psi(x_j)$ are distinct colours in $\{r-3,r-2,r-1,r\}$
Lemma~\ref{Lemma: Induction-colour-component}
and the induction hypothesis implies 
we can extend $\psi$, using at most $r$-colours, to colour all the vertices of 
\[V(P) \cap (V(G^{i,j}_{k,+}) \cup \{x_i,x_j\})\]
while satisfying 
(\ref{Lemma: Induction-colour-component (I)})-(\ref{Lemma: Induction-colour-component (VII)}) with parameter $k$, for this choice of $i,j$.
We note the choice of colours $\{r-3,r-2,r-1,r\}$ which we assignment to $x_i,x_j$ depend on the indices $i$ and $j$, since these vertices were precoloured,
as well as our choice to consider $G^{i,j}_{k,+}$ rather than $G^{i,j}_{k,-}$.

If $k = \lfloor \frac{r}{2} \rfloor -2$, then all that remains to be shown is that $\psi$ can be extended to colour $\bar{u}, \bar{v},$ and $\bar{w}$;
since in this case we only need to prove that there is an $r$-colouring of $G^{(r)}_{k+1}[V(P)\setminus \{v^A,v^B\}]$.
Recall that $N(\{\bar{u}, \bar{v},\bar{w}\}) = V(Q)$, and recall that $\psi$ uses only $4$ colours to colour $V(Q)$.
Since $r\geq 6$ there are least $2$ colours available to every vertex $\bar{u}, \bar{v},\bar{w}$.
As $\{\bar{u}, \bar{v},\bar{w}\}$ induces a graph isomorphic to $K_1+K_2$, we can extend $\psi$ to $\{\bar{u}, \bar{v},\bar{w}\}$ using $2$  of $\psi$'s colours.

Otherwise, $k \leq \lfloor \frac{r}{2} \rfloor -3$.
Hence, by condition (\ref{Lemma: Induction-colour-component (VII)}) with parameter $k$ we can conclude that for every fixed $i,j$ the vertices of
\[
V(G^{i,j}_{k,\pm}) \cap \Big(V(P) \cap \left( W^{(r)}_{k+1} \cup \{v: v \text{ is minimal in } \preceq\}\right)\Big)
\]
are all the same colour, and we can assume this colour is not in $\{r-3,r-2,r-1,r\}$.
This colour class might be inconsistent over different choice of $i,j$, since our application of Lemma~\ref{Lemma: Induction-colour-component} depended on our choice of $i,j$.
But since each pair $G^{i,j}_{k,\pm}$ and $G^{i',j'}_{k,\pm}$
is only connect via $V(Q)$ in our already coloured graph, and the vertices of $V(Q)$ use colours $\{r-3,r-2,r-1,r\}$,
we can permute the colours of each $V(G^{i,j}_{k,\pm})$, so that (\ref{Lemma: Induction-colour-component (I)})-(\ref{Lemma: Induction-colour-component (VII)}) are satisfied globally, with respect to the already coloured vertices of $P$, by $\psi$.

Suppose then that $\psi$ satisfies (\ref{Lemma: Induction-colour-component (I)})-(\ref{Lemma: Induction-colour-component (VII)}) globally with respect to the already coloured vertices of $P$. Then,
\[
\Big(V(P) \cap \left( W^{(r)}_{k+1} \cup \{v: v \text{ is minimal in } \preceq\}\right)\Big) \setminus \{\bar{w}\}.
\]
is a colour class under $\psi$.
Call this colour $t$.
Every neighbour of $\bar{w}$, other than $\bar{u}$ and $\bar{v}$, which have not been given colours yet,
has a colour in $\{r-3,r-2,r-1,r\}$.
Hence, we can let $\psi(\bar{w}) = t$, therefore 
\[
V(P) \cap \left( W^{(r)}_{k+1} \cup \{v: v \text{ is minimal in } \preceq\}\right)
\]
is a colour class of $\psi$, satisfying condition (\ref{Thm: Technical version of lower result (7)}) for parameter $k$ (not yet $k+1$ as we have yet to colour $v^A$).
Since, $k \leq \lfloor \frac{r}{2} \rfloor -3$, we note that $1 \leq \lfloor \frac{r}{2} \rfloor - k -2$.
Hence, we may let $\psi(\bar{u}) = \psi(\bar{v}) = 1 \neq t$ by condition (\ref{Lemma: Induction-colour-component (I)}) or condition (\ref{Lemma: Induction-colour-component (III)}).
The reader can verify this choice of $\psi$ satisfies conditions (\ref{Thm: Technical version of lower result (1)})-(\ref{Thm: Technical version of lower result (4)}) and condition (\ref{Thm: Technical version of lower result (7)}) for parameter $k$ (not yet $k+1$ as we have yet to colour $v^A$ and $v^B$).

Finally, with $k \leq \lfloor \frac{r}{2} \rfloor -3$,
we consider how to modify $\psi$, 
an $r$-colouring of $G^{(r)}_{k+1}[V(P)\setminus~\{v^A,v^B\}]$ which satisfies conditions (\ref{Thm: Technical version of lower result (1)})-(\ref{Thm: Technical version of lower result (4)}) and condition (\ref{Thm: Technical version of lower result (7)}) for parameter $k$ (not parameter $k+1$ yet), 
into an $r$-colouring $\phi$ of $\mathcal{G}^{(r)}_{k+1}[V(P)]$ which satisfies conditions (\ref{Thm: Technical version of lower result (1)})-(\ref{Thm: Technical version of lower result (7)}) for parameter $k+1$.
Since, $k \leq \lfloor \frac{r}{2} \rfloor -3$, we note that
\begin{align*}
    1 \leq \left\lfloor \frac{r}{2} \right\rfloor - k -2 \hspace{0.5cm}\text{ and }\hspace{0.5cm} \left\lfloor\frac{r}{2} \right\rfloor - k -1 \leq r - 2k-4.
\end{align*}
Hence, conditions (\ref{Thm: Technical version of lower result (1)})-(\ref{Thm: Technical version of lower result (4)}), with parameter $k$, for $\psi$, imply that there exists a colour $x$ and a colour $y$
such that $\psi^{-1}(x) \subseteq A^{(r)}_{k+1}$ and $\psi^{-1}(y) \subseteq B^{(r)}_{k+1}$. Let $x$ be any colour satisfying $\psi^{-1}(x) \subseteq A^{(r)}_{k+1}$,
and should one exist, let $y \neq t$ be a colour satisfying $\psi^{-1}(y) \subseteq B^{(r)}_{k+1}$.
Notice that if $k< \lfloor \frac{r}{2} \rfloor -3$, then 
\[
\left\lfloor\frac{r}{2} \right\rfloor - k -1 < r - 2k-4
\]
so condition (\ref{Thm: Technical version of lower result (2)}) or (\ref{Thm: Technical version of lower result (4)}) implies such a colour $y$ exists.
Then, if no such colour $y$ exists, $k = \lfloor \frac{r}{2} \rfloor -3$ and we let $y = t$, since condition (\ref{Thm: Technical version of lower result (7)}) for $k = \lfloor \frac{r}{2} \rfloor -3$ forces $\psi^{-1}(t) \subseteq B^{(r)}_{k+1}$.

Let $\sigma$ be any permutation such that $\sigma(x) = b$, and $\sigma(y) = a$.
We define $\phi_\sigma$ as follows: let $\phi_\sigma(v^A) =a$ and $\phi_\sigma(v^B) = b$,
for each vertex $u \in V(P) \setminus \{v^A,v^B\}$, let $\phi_\sigma(u) = \sigma (\psi(u))$.
Since, $\psi$ is an $r$-colouring of $G^{(r)}_{k+1}[P\setminus\{v^A,v^B\}]$, 
and since $\psi^{-1}(x) \subseteq A^{(r)}_{k+1}$ and $\psi^{-1}(y) \subseteq B^{(r)}_{k+1}$,
while $a\neq b$, and 
\begin{align*}
    N[v^A] = A^{(r)}_{k+1} \cup \{v^B\} \hspace{0.5cm}\text{ and }\hspace{0.5cm}  N[v^B] = B^{(r)}_{k+1} \cup \{v^B\}
\end{align*}
it is immediate that $\phi_\sigma$ is an $r$-colouring of $\mathcal{G}^{(r)}_{k+1}[P]$.

For any choice of $\sigma$ as described 
$\phi_\sigma$ will satisfy conditions (\ref{Thm: Technical version of lower result (5)}) and (\ref{Thm: Technical version of lower result (6)}) trivially.
If $k < \lfloor \frac{r}{2} \rfloor -3$, then $y\neq t$ implies condition (\ref{Thm: Technical version of lower result (7)}) is satisfied for parameter $k+1$.
If $k = \lfloor \frac{r}{2} \rfloor -3$, then $k+1 = \lfloor \frac{r}{2} \rfloor -2$, and $y = t$,
so 
\[
V(P) \cap \left( W^{(r)}_{k+1} \cup \{v: v \text{ is minimal in } \preceq\}\right)
\]
is colour class $t$ of $\psi$, which implies that 
\[
\{v^A\} \cup \Big( V(P) \cap \left( W^{(r)}_{k+1} \cup \{v: v \text{ is minimal in } \preceq\}\right) \Big)
\]
is colour class $a$ in $\phi_\sigma$, thereby satisfying condition (\ref{Thm: Technical version of lower result (7)}) for parameter $k+1$.
Since all colour classes of $\psi$ other than $x$ and $y$ are equal to colour classes in $\phi_\sigma$,
conditions (\ref{Thm: Technical version of lower result (1)})-(\ref{Thm: Technical version of lower result (4)}) applying to $\psi$ for parameter $k$ imply that
there exists a permutation $\sigma'$ (satisfying the same conditions as $\sigma$) such that
conditions (\ref{Thm: Technical version of lower result (1)})-(\ref{Thm: Technical version of lower result (4)}) apply for parameter $k+1$ in $\phi_{\sigma'}$.

Let $\sigma'$ be such a permutation, and let $\phi = \phi_{\sigma'}$.
Then, $\phi$ is an $r$-colouring of $\mathcal{G}^{(r)}_{k+1}[V(P)]$ satisfying conditions (\ref{Thm: Technical version of lower result (1)})-(\ref{Thm: Technical version of lower result (7)}).
This completes the proof.
\end{proof}

\section{Future Work}
\label{sec: Future Work}

To conclude the paper we 
propose some conjectures and open problems for future work.
Of course, the principle problem in this area remains to resolve Conjecture~\ref{conjecture: r=3,k=4},
however this seems out of reach.
Instead we focus on introducing problems that seem more approachable, although in the same vein.

Extending 
our work from Section~\ref{sec: star-free} and Section~\ref{sec: small H-free}:

\begin{problem}
    What is the least integer $t$ such that there exists a $K_{1,t}$-free graph $G$ where $\chi(G) > r(G)$?
\end{problem}

Phrased differently, is every claw-free graph path-perfect?
If yes, then what is the greatest integer $\ell$ such that every $K_{1,\ell}$-free graph is path-perfect?
Notice that the left graph in Figure~\ref{fig:Gallai3-4Critical} has $\chi> r$ and is $K_{1,6}$-free. 

Similarly, given Lemma~\ref{Lemma: Ind = 1,2,3}, one can ask if graphs with independence number $4$ are path-perfect.

\begin{problem}
    Is every graph $G$ with $\alpha(G) = 4$ path-perfect?
\end{problem}

We also conjecture that Theorem~\ref{Thm: Small Forest Free} can be strengthened.

\begin{conjecture}
    If $G$ is $2K_2$-free or $(K_2+2K_1)$-free, then $G$ is path-perfect.
\end{conjecture}

Moreover, we believe that an even stronger statement could be true.

\begin{conjecture}
    If $G$ is $P_5$-free graph, then $G$ is path-perfect. 
\end{conjecture}

Of course one can also ask, what is the least integer $t$, such that there exists a $P_t$-free graph $G$ with $\chi(G) > r(G)$?
In this case, it is worth noting that the
left graph in Figure~\ref{fig:Gallai3-4Critical} has $\chi> r$ and is $P_{10}$-free. 

Our final problem relating to graphs with a forbidden induced subgraph, is a further weakening of Hajnal and Erd\H{o}s' problem.
This weakening is of interest, since it marks a stark difference $\chi$-boundedness and bounding chromatic number by a function of $r$.
It is not obvious to the authors what a good candidate for the graph $H$ should be.

\begin{conjecture}
    There exists a graph $H$ and a function $h = h_H$ 
    such that $H$ contains a cycle and if $G$ is $H$-free, 
    then $\chi(G) \leq h(r(G))$.
\end{conjecture}

To conclude we conjecture that our main result, Theorem~\ref{Thm: Lower result}, can be improved as follows.

\begin{conjecture}
    For every constant $c>0$ there exists a graph $G$ with
    \[
    \chi(G) > c\cdot r(G).
    \]
\end{conjecture}

\section*{Acknowledgement}

The authors would like to thank Ben Seamone
for introducing us to Gy\'{a}rf\'{a}s' conjecture in personal correspondence with Clow.
The conjecture captured our attention, despite our unsuccessful attempt to prove it, and we do not believe we would have encountered it otherwise.
Ben Cameron gratefully acknowledges support from the Natural Sciences and Engineering Research Council of Canada (NSERC), grants RGPIN-2022-03697 and DGECR-2022-00446.

\bibliographystyle{abbrv}
\bibliography{bib}

\end{document}